\documentclass[12pt]{amsart}
\usepackage{mathrsfs}
\usepackage{eucal}
\usepackage{xcolor}
\usepackage[all]{xy}
\usepackage{amssymb,amsmath}
\usepackage{enumerate}
\usepackage{relsize}
\usepackage{comment}
\usepackage{listings}
\usepackage{mathtools}

\newdir^{ (}{{}*!/-8pt/@^{(}}% % A new tip style for right hook arrows.
\usepackage{hyperref}
\hypersetup{
	colorlinks=true,
	linkcolor=blue,
	filecolor=blue,
	citecolor = blue,
	urlcolor=blue,
}

\date{February 3, 2026}
\calclayout
\newtheorem{dummy}{anything}[section]
\newtheorem{theorem}[dummy]{Theorem}
\newtheorem*{thma}{Theorem A}
\newtheorem*{thmb}{Theorem B}
\newtheorem*{thmc}{Theorem C}
\newtheorem*{thme}{Theorem E}

\newtheorem*{cord}{Corollary D}

\newtheorem{lemma}[dummy]{Lemma}
\newtheorem{proposition}[dummy]{Proposition}
\newtheorem{corollary}[dummy]{Corollary}
\newtheorem{conjecture}[dummy]{Conjecture}
\theoremstyle{definition}%%Change Theoremstyle
\newtheorem{definition}[dummy]{Definition}
  \newtheorem{example}[dummy]{Example}
  
  \newtheorem{remark}[dummy]{Remark}
   
    \newtheorem*{question}{Question}
    \newtheorem*{oozeh}{Oozing Conjecture}
    
  \newtheorem*{acknowledgement}{Acknowledgement}

\newcommand
{\eqncount}{\setcounter{equation}{\value{dummy}}%
\addtocounter{dummy}{1}}
\newcommand{\cB}{\mathcal B}

\newcommand{\cI}{\mathscr I}

\newcommand{\bF}{\mathbf F}

\newcommand{\bZ}{\mathbb Z}
\newcommand{\bQ}{\mathbf Q}

\newcommand{\bR}{\mathbb R}

\newcommand{\bbL}{\mathbb L}

\newcommand{\bbQ}{\mathbb Q}

\DeclareMathOperator{\Hom}{Hom}
\DeclareMathOperator{\Image}{Im}

\DeclareMathOperator{\invlim}{\lower 6pt\hbox{$\stackrel{\displaystyle{\lim}}{\leftarrow}$}}
\DeclareMathOperator{\dirlim}{\lower 6pt\hbox{$\stackrel{\displaystyle{\lim}}{\rightarrow}$}}

\newcommand{\mmatrix}[4]{\left (\vcenter
{\xymatrix@C-2pc@R-2pc{#1&#2\\#3&#4} }
\right )}

\newcommand{\cy}[1]{\bZ/{#1}}
\newcommand{\la}{\langle}
\newcommand{\ra}{\rangle}

\newcommand{\vv}{\, | \,}
\newcommand{\bd}{\partial}

\def\Ztwo{\bZ_{(2)}}
\def\Zhat{\widehat\bZ_2}

\DeclareMathOperator{\Mod}{(mod)}

\DeclareMathOperator{\Wh}{Wh}

\newcommand{\Zpi}{\bZ\pi}

\newcommand{\wkappa}{\widetilde\kappa}
\newcommand{\trf}{{tr{\hskip-2pt}f}}
\newcommand{\Rpi}{R\pi}
\newcommand{\Rhat}{\widehat R_2}
\newcommand{\wkappas}{\tilde\kappa^s}

\newcommand{\wx}{\widehat x}
\newcommand{\wz}{\widehat z}

\usepackage[symbol]{footmisc}

\newcommand{\wpi}{\widetilde\pi}
\newcommand{\ab}{\textup{ab}}

\begin{document}

%\title{Closed manifold surgery obstructions revisited}
\title{Closed manifold surgery obstructions and the Oozing Conjecture}

\author{Ian Hambleton}

\address{Department of Mathematics, McMaster University,  Hamilton, Ontario L8S 4K1, Canada}

\email{hambleton@mcmaster.ca }

\author{\"Ozg\"un \"Unl\"u}
\address{Department of Mathematics,
Bilkent University,
Ankara, Turkey 06800}
\email{unluo@fen.bilkent.edu.tr}

\thanks{Research partially supported by NSERC Discovery Grant A4000. The second author was supported by the Fields Institute for Research in Mathematical Sciences, under a Fields Research Fellowship held at McMaster in July 2024.
}

%s\subjclass[2020]{Primary 57R67, 20J06; Secondary 18F25}

\begin{abstract}
We  complete the description of surgery obstructions up to  homotopy equivalence for closed oriented manifolds with finite fundamental group. New examples are presented of non-trivial obstructions for Arf invariant product formulas in codimensions $\geq 4$, which give counterexamples to the well-known ``Oozing Conjecture'' from the 1980's. 
\end{abstract}

\maketitle

\section{Introduction}
\label{sect: introduction}
The surgery exact sequence \cite[Chapter 9]{Wall:1970} provides a framework for the classification of manifolds and their automorphisms. 
In this theory, the relationship between the topology of manifolds and homotopy theory is analysed by studying certain  \emph{assembly maps}  in order to compute the surgery obstructions of degree 1 normal maps $(f, b)\colon N^n \to M^n$, where the domain and range are \emph{closed manifolds} (see \cite{Quinn:1970}, \cite{Ranicki:1992}, \cite{Hambleton:2004}). These obstructions take values in algebraically defined \emph{surgery obstruction groups} $L_n(\Zpi,w)$, depending  only on the fundamental group $\pi:= \pi_1(M)$, the orientation character $w\colon \pi \to \{\pm 1\}$, and  $n  = \dim M$ (mod 4). The image of the $L$-theory assembly map is the subgroup of  $L_n(\Zpi,w)$ consisting of the closed manifold surgery obstructions (see \cite[Theorem 13B.3]{Wall:1970}, \cite[\S 2]{Hambleton:2012}).  

In this paper we will restrict to surgery problems for oriented manifolds with finite fundamental group. We note that the celebrated  Farrell-Jones conjectures \cite{Farrell:1993} provide the best information to date about
the $L$-theory assembly maps for infinite fundamental groups, but give no information for finite groups (see \cite{Luck:2020} for a recent survey on progress towards the Farrell-Jones conjectures, and \cite{Weinberger:1982} for a striking example).

\smallskip
The surgery obstructions up to simple homotopy equivalence for closed manifolds with finite fundamental group $\pi$ are determined by transfer to the $2$-Sylow subgroup (see \cite{Hambleton:1988}). The cohomological formula of Wall \cite{Wall:1976} and Taylor-Williams \cite{Taylor:1979} reduces the computation of the surgery obstruction map
$$[X, G/TOP] \to L^s_n(\bZ \pi)$$
to some characteristic class information and the determination of two families of universal homomorphisms:
$$\cI^s_j(\pi)\colon H_j(\pi; \Ztwo) \to L^s_j(\bZ \pi)_{(2)}$$
and
$$\kappa^s_j (\pi)\colon H_j(\pi;\cy 2) \to L^s_{j+2}(\bZ \pi)_{(2)}$$
defined for all $j \geq 0$. Note that the calculation of these homomorphisms is equivalent to determining the $L$-theory assembly map
$$ A_{\pi} \colon B\pi^+ \wedge \bbL_0(\bZ) \to \bbL_0^s(\Zpi)$$
for oriented manifolds with finite fundamental group (see \cite[\S 1]{Hambleton:1988}).

\smallskip
We recall that for oriented manifolds  $\cI^s_0$ may be identified with the ordinary signature, and $\kappa^s_0$ with the ordinary Arf-Kervaire invariant (defined by projection to $\pi=1$). From now on, we assume that all $L$-groups are equipped with the standard oriented involution ($g \mapsto g^{-1}$) and localized at 2. By \cite[Theorem 7.4]{Wall:1974} this is no loss of information,  since the torsion in $L$-theory is all $2$-primary for $\pi$ a finite group.

More generally, let $U \subseteq Wh(\Zpi)$ be an involution invariant subgroup. Then one can consider the obstructions to surgery up to homotopy equivalence with torsions lying in $U$. The most important examples are $U= \{0\}$ (simple), $U = SK_1(\Zpi)$ (weakly simple), and $U = Wh(\Zpi)$, with corresponding obstruction groups $L^s$, $L'$, and $L^h$ respectively. 

\smallskip
The closed manifold surgery obstructions in $L^U_*(\bZ \pi)$, for surgery up to \emph{weakly simple} homotopy equivalence, were substantially determined 
in \cite[Theorem A]{Hambleton:1988}, for any Whitehead torsion decoration
$$SK_1(\bZ \pi) \subseteq U \subseteq \Wh(\bZ \pi).$$
More precisely, all the weakly simple universal homomorphisms were computed except for $\kappa^U_4$. In particular, $\cI^{U}_j = 0$, for $j>0$ (see 
\cite[Theorem 2.1]{Hambleton:1988} for a slightly sharper result).

\smallskip
For surgery up to simple homotopy equivalence, we do not yet have a general computation of the universal homomorphisms. Taylor and Williams \cite[Theorem 4.1]{Taylor:1980} generalized work of Stein \cite{Stein:1976} on product formulas for surgery obstructions to study the universal homomorphisms for finite fundamental groups with \emph{special} $2$-Sylow subgroups (cyclic, dihedral, semi-dihedral and quaternion), and obtained most of the information for abelian $2$-Sylow subgroups. 

\smallskip
We  address some of the open questions concerning the universal homomorphisms. Our main results are:
\begin{enumerate}
\item A complete determination of the surgery assembly map in terms of group homology, for surgery up to homotopy equivalence (Theorem E).
\item A complete description of simple closed manifold surgery obstructions for abelian or basic  $2$-groups (Theorem A).
\item A description of the image of simple closed manifold surgery obstructions under the change of coefficients map $L^s_*(\bZ \pi) \to L^s_*(\Zhat \pi)$ (Theorem B and Theorem C).
\item Examples showing that $\ker \kappa^h_2 \neq \ker \kappa^s_2$ (Corollary D).
\item   An example showing that the homomorphism $\kappa^h_4$ can be non-zero (see Corollary \ref{cor:ooze_example}). 
\end{enumerate}

\begin{remark} These results are obtained by further exploiting a key ingredient in our joint paper \cite{Hambleton:1988}, namely a lifting of the $\{\kappa^s_*\}$ through 
the $L$-theory of a quadratic extension ring $R = \bZ[\varepsilon]$, with $\varepsilon = (1+ \sqrt{5})/2$. There is a natural factorization
$$\xymatrix@C-5pt{H_j (\pi;\cy 2) \ar[dr]_{\wkappas_j}\ar[rr]^{\kappa^s_j} && 
L_{j+2}^s(\Zpi)_{(2)}\\
& L_j^s(R\pi)\ar[ur]_{\trf}&
}$$
where $\trf$ is the induced transfer map on $L$-theory (see \cite[Theorem 1.16]{Hambleton:1988}). 
 \end{remark}
 
 Two important examples, due to Morgan-Pardon (unpublished, \cite[Theorem 4.6]{Milgram:1986}) and Cappell-Shaneson \cite{Cappell:1979}, showed the non-triviality of certain product formulas
$$ (f\times 1, b\times 1)\colon K^{4n+2} \times M^k \to S^{4n+2} \times M^k$$
for  the surgery obstructions in $L^h_{k+2}(\bZ \pi)$. In these products,  $(f,b)\colon K^{4n+2}  \to S^{4n+2}$ denotes a degree 1 normal map with domain the closed topological Kervaire manifold with Arf-Kervaire invariant 1, and $M^k$ is a closed, oriented manifold with $\pi_1(M, x_0) =\pi$. We recall that these obstructions vanish if $k \equiv 0 \pmod 4$ and the Euler characteristic of $M$ is even (see \cite[Theorem (0.2)(i)]{Hambleton:1988}). 

\smallskip
More generally, one can also  consider the oriented surgery obstructions of twisted products of  the Arf invariant normal map $(f,b)\colon K^{4n+2}  \to S^{4n+2}$ with  non-orientable manifolds    (see \cite[\S 4]{Hambleton:2012}).

 The surgery obstruction up to homotopy equivalence of such a (twisted) product normal map, modulo the image  of the ordinary Arf-Kervaire invariant, is sometimes called a \emph{codimension $k$ Arf invariant}. A well-known conjecture from the 1980's stated:
\begin{oozeh} The  \emph{codimension $k$} Arf invariants are trivial  for $k\geq 4$.
\end{oozeh}
The Morgan-Pardon example shows that $\kappa_2^h \neq 0$, and  the Cappell-Shaneson example shows that $\kappa_3^h \neq 0$. In general, 
the oozing conjecture is equivalent (by the cohomological formula for surgery obstructions) to the vanishing of all the homomorphisms $\kappa^h_j$,  for $j\geq 4$ (see \cite[Theorem A]{Hambleton:1988}).  However, our Theorem E below and the example in Corollary \ref{cor:ooze_example} shows that the oozing conjecture in this form is false.

\medskip
There is a sharper version of this conjecture  (also disproved by Corollary \ref{cor:ooze_example}), concerning the oriented surgery obstructions up to \emph{simple} homotopy equivalence of  twisted products with the Arf invariant normal map $(f,b)\colon K^{4n+2}  \to S^{4n+2}$.  The vanishing of these obstructions in $L_{k+2}^s(\Zpi)$ in codimension $k \geq 4$ is a stronger statement, equivalent to the vanishing of the homomorphisms $\kappa^s_j$,  for $j\geq 4$. Here is a natural question.

\begin{question}\label{question:kappas} Does there exist a finite $2$-group $\pi$ such that  $\kappa^s_4(\pi) \neq 0$, but  $\kappa^h_4(\pi)= 0$ ?
\end{question} 
In Section \ref{subsec:codim_two_ex}  we present an approach to answer this question positively through further GAP computations, and provide an example in Proposition \ref{prop:kappas_ex2}.

\subsection{Results for basic or abelian $2$-groups}
Here are the results for finite abelian $2$-groups and the \emph{basic} $2$-groups (cyclic, quaternion, dihedral, and semi-dihedral). These cases are important as building blocks for determining the surgery obstructions over more general finite $2$-groups (compare \cite[Theorem 0.3]{Hambleton:1988}).

\begin{thma}[\cite{Taylor:1980}, \cite{Hambleton:1988}]
Suppose that $\pi$ is an abelian or basic finite $2$-group. Then
\begin{enumerate}
\item $\cI^s_j = 0$, for $j >0$;
\item $\kappa^s_0$ and $\kappa^s_1$ are (split) injective;
\item $\kappa^s_2 = 0$ for the basic $2$-groups;
\item $\kappa^s_j = 0$, for $j >2$, unless $\pi$ is quaternion;
\item If $\pi$ is quaternion, $\kappa^s_3$ is injective and $\kappa_j^s = 0$, for $j \geq 4$.
\end{enumerate}
Moreover, if $E\subseteq \pi$ denotes the subgroup of an abelian $2$-group $\pi$ consisting of elements of order $\leq 2$, then the sequence
$$H_2(E;Z/2) \to H_2(\pi;\bZ/2) \xrightarrow{\kappa^s_2} L^s_0(\bZ \pi)$$
is exact. In addition, $\ker \kappa^s_{2}(\pi) = \ker \kappa^h_{2}(\pi)$ and their images are isomorphic.
\end{thma}
\noindent
{\bf Notation}:  we write $H^n(A): = \widehat H^n(\cy 2; A)$ for Tate cohomology with coefficients in a $\cy 2$-module $A$.
\begin{remark} These results refine those proved in Taylor-Williams \cite{Taylor:1980}: we will show how they can be derived from the information in \cite{Hambleton:1988}. The difference between $\kappa_*^s$ and $\kappa_*^h$ is partly measured by the intermediate comparison sequence
$$ \dots \to H^{n+1}(SK_1(\Zpi)) \xrightarrow{\ \bd \ } L^s_n(\Zpi) \to  L'_n(\Zpi)  \to H^{n}(SK_1(\Zpi))  \to \dots$$
where $L'_*(\Zpi)$ denotes the \emph{weakly simple} $L$-theory with torsion decorations in 
$\Wh'(\Zpi) := \Wh(\Zpi)/SK_1(\Zpi)$. 
The passage from $L_*'$ to $L^h_*$ is given by the exact sequence
$$0 \to L'_{2k}(\Zpi) \to L^h_{2k}(\Zpi) \to \Wh'(\Zpi)\otimes \cy 2 \to  L'_{2k-1}(\Zpi) \to  L^h_{2k-1}(\Zpi) \to 0$$
due to Wall \cite[p.~78]{Wall:1976a}. In particular, this shows that in even dimensions $\kappa_{2j}^h$ determines $\kappa'_{2j}$ for any finite group (meaning that $\ker \kappa'_{2j} = \ker \kappa^h_{2j}$ and their images are isomorphic). 

The last part of Theorem A implies that a stronger statement holds for $\pi$ a finite abelian $2$-group, namely that the image of $\kappa^s_2(\pi)$ has zero intersection with the image of the natural map
$\bd\colon H^1(SK_1(\Zpi)) \to L_0^s(\Zpi)$.
\end{remark}

\subsection{Results for the $2$-adic homomorphisms}

Next we concentrate on the compositions:
$$\bar\kappa^s_j\colon H_j(\pi;\cy 2) \to L^s_{j+2}(\bZ \pi) \xrightarrow{\ \rho_2\ } L^s_{j+2}(\Zhat \pi)$$
of the universal homomorphisms under the map induced by $\rho\colon \bZ \to \Zhat$ into the 2-adic $L$-groups. Recall that calculations of the surgery obstructions are based on separating the $2$-adic information from the integral information via an exact sequence of the form
$$ \dots \to L_{j+1}(\Zpi \to \Zhat\pi) \to  L_{j}(\Zpi) \to L_j(\Zhat\pi) \to \dots$$
with appropriate torsion decorations (see Wall \cite[1.2]{Wall:1976a}). The relative $L$-groups have better induction and detection properties (see \cite[2.1]{Wall:1976a}, \cite{Hambleton:1990}), which we can exploit to obtain information about the integral $\kappa_*$-homomorphisms from their $2$-adic images (see Lemma \ref{lem:threethree}).  In the next result let $s_r \colon H_{2^{r+2}}(\pi; \cy 2) \to H_4(\pi; \cy 2)$, for $r >0$, denote the Hom-dual of the iterated squaring maps in cohomology.

\begin{thmb} For any finite $2$-group $\pi$, the 2-adic universal homomorphisms 
$$\bar\kappa^s_j\colon H_j(\pi;\cy 2) \to L^s_{j+2}(\Zhat \pi)$$
for $j \neq 2$ are determined as follows:
\begin{enumerate}
\item $\bar\kappa^s_0$ and $\bar\kappa^s_1$ are split injective. 
\item If $j>0$, then $\bar\kappa_j^s = 0$ provided $j \neq 2^l$, for $l \geq 1$. 
\item $\bar\kappa^s_{2^{r+2}} = \bar\kappa^s_4\circ s_r$ for $r \geq 0$, and $\bar\kappa^s_i$ vanishes on the image of integral homology for $i \geq 4$.
\end{enumerate}
Moreover, the image of $\kappa'_{j}\colon H_{j}(\pi;\cy 2) \to L'_{j+2}(\Zpi)$, for $j \equiv 0\pmod 2$,  is detected by the natural map  
$L_{j+2}'(\Zpi) \to L_{j+2}'(\Zhat\pi)\oplus L_{j+2}'(\Zpi^{\ab})$.
\end{thmb}
\begin{remark} The homomorphism $\bar\kappa'_2\colon  H_2(\pi;\cy 2) \to L'_{0}(\Zhat\pi)$ is computed in Proposition \ref{prop:kappatwoprime}. In particular, $\kappa'_2(\pi)$ is non-zero if and only if $H^1(\Wh'(\Zhat\pi)) \neq 0$.
In Appendix \ref{sec:appone} we list the $2$-groups of order $\leq 128$ for which $H^1(\Wh'(\Zhat\pi)) \neq 0$  (a necessary condition is that $\pi$ is non-abelian). See Theorem C for more information about $\bar\kappa^s_2$. 
\end{remark}

\subsection{Results for surgery up to  simple homotopy equivalence}
We now settle some open questions from \cite{Hambleton:1988}. Recall that there is a boundary map $\bd\colon H^1(SK_1(\Zhat\pi)) \to L_0^s(\Zhat\pi)$ in the Ranicki-Rothenberg sequence comparing $L^s_*(\Zhat\pi))$ to $L'_*(\Zhat\pi)$ under the change of Whitehead torsion decorations.
\begin{thmc} Let $\pi$ be a finite 2-group. Suppose that the following conditions hold:
\begin{enumerate}
\item the boundary map $\bd\colon H^1(SK_1(\Zhat\pi)) \to L_0^s(\Zhat\pi)$ is non-zero.
\item $H^1(\Wh'(\Zhat\pi)) = 0$.
\end{enumerate}
Then $\kappa^s_2(\pi)$ is non-zero 
 and $\bar\kappa'_2(\pi) = 0$. If $H^1(SK_1(\Zhat\pi)) \to L_0^s(\Zhat\pi)$ is injective, then $\Image\{\bar\kappa^s_2(\pi)\colon H_2(\pi;\cy 2) \to L_0^s(\Zhat\pi)\} \cong H^1(SK_1(\Zhat\pi))$.
\end{thmc}

\begin{cord} There exist finite $2$-groups $\pi$ such that $\ker \kappa^s_2(\pi) \subsetneq \ker\kappa^h_2(\pi)$.
\end{cord}

\begin{remark} This answers a question of Kasprowski, Nicholson and Vesel\'a \cite[1.5]{Kasprowski:2024}, 
where it is shown in   \cite[Theorem C]{Kasprowski:2024} that the existence of such examples implies  the existence of closed smooth 4-manifolds which are homotopy equivalent but not simple homotopy equivalent (even up to stabilisations). 
\end{remark}

In Section \ref{sec:three} we give a number of examples of groups $\pi$ for which the conditions in Theorem C are satisfied (see Examples \ref{ex:thmEgroups}). Among these groups, we have examples for Corollary D  in which $\ker \kappa^s_2(\pi) \subsetneq \ker\kappa^h_2(\pi)$, contrasting with the Morgan-Pardon example:  $\pi = \cy 2 \times \cy 4$, where $\kappa^h_2(\pi) \neq 0$. In Section \ref{subsec:k2nonzero}  we give details for the group $\pi = SG[128, 1377]$ in the Small Groups Library, where the notation $SG[i,j]$ means the $j^{th}$ group of order $i$ (see the list in  Examples \ref{ex:corDgroups}).

\subsection{Results on the Oozing Conjecture}
Finally, we reduce the original ``oozing conjecture" for surgery up to homotopy equivalence to an explicit calculation in group homology,  and present a counter-example.

In the following statement, let $s_*\colon H_4(\pi;\cy 2) \to H_2(\pi; \cy 2)$ denote the dual homomorphism to $Sq^2 \colon H^2(\pi;\cy 2) \to H^4(\pi;\cy 2)$, and $\beta\colon H_2(\pi;\cy 2) \to H^0(\pi^{\ab})$ the map induced by the Bockstein. In addition, the short exact sequence
$$0 \to \Wh'(\Zhat\pi) \to \overline{I(\Zhat\pi)} \to \pi^{\ab} \to 0$$
induces a (surjective) coboundary map $\delta\colon H^0(\pi^{\ab}) \to  H^1(\Wh'(\Zhat\pi))$. For the notation see \cite[p.~163]{Oliver:1988a}.

\begin{thme} Let $\pi$ be a finite 2-group, and let 
$$\lambda_4(\pi) := \delta\circ \beta\circ s_* \colon  H_4(\pi;\cy 2) \to H^1(\Wh'(\Zhat\pi)).$$
Then the homomorphism $\kappa'_4(\pi)\neq 0$ if and only if $\lambda_4(\pi) \neq 0$. Moreover, $\kappa'_4(\pi)$ is zero on the image of the integral homology $H_4(\pi; \bZ) \to H_4(\pi; \cy 2)$.
\end{thme}
 An example showing that there exists a finite $2$-group of order $16384$ with $\kappa'_4(\pi) \neq 0$ (and hence $\kappa^s_4 (\pi) \neq 0$) is presented in Section \ref{subsec:fivethree}, Corollary \ref{cor:ooze_example}. Computer investigations suggest that the smallest such examples would have orders $\geq 512$.
 
\begin{corollary}\label{cor:ooze} The oozing conjecture is false for surgery up to simple homotopy equivalence.
There exist non-trivial codimension $k$ Arf invariant problems in codimensions $k = 2^l$, for $l\geq 2$.
\end{corollary}

\begin{remark} In Section \ref{sec:five}, we derive the  explicit formula in Theorem E via
the commutative diagram (see Proposition \ref{prop:kappafourprime}):
$$\xymatrix@C+7pt{
H_4(\pi;\cy 2) \ar[r]^(0.4){\lambda_4(\pi)} \ar[d]_{\ \kappa'_4(\pi)\ \ } & H^1(\Wh'(\Zhat\pi))\ar[d]^{\bd}\\
L'_2(\Zpi) \ar[r]^{\rho_2} & L'_2(\Zhat\pi)
}$$
in which the boundary map $\bd$ is injective. This reduces the original oozing conjecture that $\kappa^h_j = 0$, for $j \geq 4$, to an explicit group homology calculation, namely computing the map $\lambda_4(\pi)$.
Recall that there is an injection $L'_2(\Zpi)  \to L^h_2(\Zpi) $ and hence $\kappa^h_4 \neq 0$ if and only if $\kappa'_4\neq 0$, so it is enough to determine $\kappa'_4(\pi)$.
\end{remark}

\setcounter{tocdepth}{1}
 \tableofcontents
 \begin{acknowledgement}  We would like to thank Larry Taylor for recent helpful conversations, and for all our joint work over the years with our friend and collaborator, the late Bruce Williams\footnote[2]{August 17, 1945 - January 11, 2018}. We also thank David Green and Simon King for useful information about  the cohomology of finite $2$-groups, and Bob Oliver for help with $SK_1(\Zhat\pi)$.
\end{acknowledgement}
%\setcounter{tocdepth}{1}
% \tableofcontents
\section{Abelian or basic $2$-groups}\label{sec:one}
 We begin by summarizing some results about the $K$-theory of group rings (a convenient reference and guide to the literature is \cite{Oliver:1988a}).
 Recall that there is a short exact sequence 
$$0 \to SK_1(\Zpi) \to \Wh(\Zpi) \to \Wh'(\Zpi) \to 0$$
where $SK_1(\Zpi) = \ker\{K_1(\Zpi) \to K_1(\bbQ\pi)\}$, and 
$$\Wh(\Zpi) = K_1(\Zpi)/SK_1(\Zpi) \oplus \{\pm 1\} \oplus \pi^{\ab}.$$
 The quotient $\Wh'(\Zpi)$ is finitely generated and torsion free, of $\bZ$-rank equal to the difference between the number of irreducible real and rational representations of $\pi$.
In addition, we have a short exact sequence
$$0 \to Cl_1(\Zpi) \to SK_1(\Zpi) \to SK_1(\Zhat\pi) \to 0$$
where $SK_1(\Zhat\pi) = \ker\{K_1(\Zhat\pi) \to K_1(\hat\bbQ_2\pi)\}$.

\smallskip
The Tate cohomology of these various $K$-groups appear in Rothenberg sequences relating the variant forms of $L$-theory, namely $L^s$, $L'$ and $L^h$. 
Here are some of the facts we need (for $\pi$ of $2$-power order):
\begin{itemize}
\item \cite[14.2; 14.4 ]{Oliver:1988a} $SK_1(\Zpi) = 0$ if $\pi$ is basic or elementary abelian.
\item  \cite[8.6]{Oliver:1988a} There is a natural isomorphism $\Theta_\pi\colon SK_1(\Zhat\pi) \cong H_2(\pi; \bZ)/H_2^{\ab}(\pi)$. The involution $g \mapsto g^{-1}$ induces multiplication by $-1$ on $SK_1(\Zhat\pi)$. If $\pi$ is abelian, $SK_1(\Zpi) = Cl_1(\Zpi)$.
\item \cite[Prop.~13]{Oliver:1980a} $H^1(\Wh'(\Zhat\pi)) = 0$ if $\pi$ is basic or abelian.
\item \cite[8.7]{Oliver:1988a} The transfer $\trf\colon SK_1(\Rhat\pi) \to SK_1(\Zhat\pi)$ is an isomorphism.
\end{itemize}

\begin{proof}[The proof of Theorem A] Since $SK_1(\Zpi) = 0$ for any basic or abelian $2$-group, we have $L^s_*(\Zpi) = L'_*(\Zpi)$ and $\kappa^s_*(\pi) = \kappa'_*(\pi)$.
Then part (i) follows immediately from  \cite[Theorem 2.1]{Hambleton:1988}: this is essentially the argument in \cite[Theorems 1.2 \& 4.1]{Taylor:1980}.

 Part (ii) follows from naturality and \cite[Theorem 0.3(ii)]{Hambleton:1988}, and the vanishing results of
 parts (iv) and (v) follow from \cite[Theorem 6.8]{Hambleton:1988} by applying the $L$-theory transfer. For $\pi$ quaternion, the Cappell-Shaneson example shows that $\kappa^s_3(\pi)$ is injective. 
 
Next we consider  part (iii). From the calculations of Wall \cite[\S 5.2]{Wall:1976a}, $L_0'(\Zpi)$ is torsion free if $\pi$ is a basic $2$-group, since $H^1(\Wh'(\Zhat\pi)) = 0$. This gives  $\kappa^s_2(\pi)  = \kappa'_2(\pi) = 0$.

It remains to consider $\kappa^s_2(\pi)=\kappa'_2(\pi) $ for  $\pi$ abelian. 
We can
write $\pi = C_1 \times C_2 \times \dots \times C_r$, where $C_i \cong \cy{2^{k_i}}$. Then there is a surjection
$\bigoplus H_1(C_i; \cy 2) \to H_1(\pi; \cy 2)$, and 
\eqncount
\begin{equation}\label{eq:homtwo} H_2(\pi;\cy 2) \cong \bigoplus_{i = 1}^r H_2(C_i; \cy 2) \oplus \bigoplus_{i <j} H_1(C_i; \cy 2) \otimes H_1(C_j; \cy 2).
\end{equation}
In general, if $\pi$ is an abelian $2$-group of rank $r$ with $s$ summands of order two, then
\eqncount
\begin{equation}\label{eq:lgroup}
L'_0(\Zpi) = \Sigma \oplus \left (2^r -r -1 - {s\choose 2}\right )\cdot \cy 2
\end{equation}
by \cite[Theorem 3.3.2]{Wall:1976a}. In particular, $L'_0(\Zpi)$ is torsion free for $\pi $ cyclic or $\pi = C_2 \times C_2$. 

Suppose first that $\pi$ is elementary abelian of rank $\geq 2$, and let $\gamma  = C_2 \times C_2$ be a subgroup of $\pi$, with $i_*\colon \gamma  \to \pi$ the inclusion map. Then $$\kappa'_2(i_*(x)) = i_*(\kappa'_2(x)) \in \Image\{L'_0(\bZ\gamma ) \xrightarrow{\ i_*\ } L'_0(\Zpi)\}.$$
 But $\kappa'_2(\gamma ) = 0$ since $L'_0(\bZ\gamma )$ is torsion free, hence $\kappa'_2(i_*(x)) = 0$. Since every class in $H_2(\pi;\cy 2)$ is a sum of classes induced from subgroups of rank $\leq 2$, and $\kappa_2^s=0$ for cyclic groups, it follows that $\kappa^s_2(\pi) =0$ for elementary abelian $2$-groups.

For $\pi$ any abelian group, we can again restrict to elements $x \in H_2(\pi;\cy 2)$ with $\kappa_2'(x) = 0$ which are a sum of classes induced from rank 2 subgroups. If $\pi$ has rank $r$ and $s$ summands of order two, we can write $\pi = \pi_1 \times \pi_2$, where 
$\pi_1 = C_ \times \dots \times C_s$ is elementary abelian, and $\pi_2 = C'_1 \times \dots C'_{r-s}$ consists of the factors with order $\geq 4$.
We denote by $\{\alpha_i\}$ the generators of $H_1(C_i;\cy 2)$, and by $\{\beta_j\}$ those of  $H_1(C'_j;\cy 2)$.
By formula  \eqref{eq:homtwo} we may assume  that 
$$x = \sum a_{ij} (\alpha_i \otimes \beta_j) + \sum b_{kl} (\beta_k \otimes \beta_l) \in H_2(\pi; \cy 2)$$
modulo the image of $H_2(E;\cy 2)$.
Note that these classes are the reductions of integral homology classes in $H_2(\pi;\bZ)$.
For each subgroup $\gamma = C_i \times C'_j$ or $\gamma = C'_k \times C'_l$,  we have
$H_2(\gamma ;\bZ) = \cy 2$, and the non-triviality of the Morgan-Pardon example in $L_0'(\cy 2 \times \cy 4)\cong \cy 2$ shows by projection that $\kappa_2'(\gamma )$ is an isomorphism for all such groups.

  Since there exist independent projections $\{p_\gamma \colon \pi \to \cy 2 \times\cy{4}\}$,  such that $p_\gamma(\gamma') = 1$ for $\gamma' \neq \gamma$, it follows by naturality that $x \equiv 0 \pmod{H_2(E;\cy 2)}$ and exactness is proved.

\medskip
To establish the last part of Theorem A, consider the diagram
$$\xymatrix{H_2(E;Z/2)\ar[d]_{i_*}\ar[r]^{\kappa_2^s}&L^s_0(ZE)\ar[r]^{\cong}\ar[d]&L'_0(ZE)\ar[d]^{i_*}\cr
H_2(\pi;Z/2)\ar@/_2pc/[rr]^{\kappa_2'}\ar[r]^{\kappa_2^s}&L^s_0(\Zpi)\ar[r]&L'_0(\Zpi) }$$
 If $\kappa_2^s(x) \neq 0$ for some $x \in H_2(\pi;\cy 2)$,  but $\kappa_2'(x) = 0$, then there exists $y \in H_2(E;\cy 2)$ such that $i_*(y) = x \in H_2(\pi;\cy 2)$. But then $\kappa_2^s(x) = \kappa^s_2(i_*(y)) = i_*(\kappa^s_2(y)) = 0$, since $\kappa^s_2 = 0$ for elementary abelian $2$-groups. If $\kappa_2'(x) \neq 0$, then  the exact sequence
$$ \dots \to H^1(SK_1(\Zpi)) \to L^s_0(\Zpi)\to L'_0(\Zpi)\to \dots$$
shows that the image of $\kappa^s_2(\pi)$ has zero intersection with the image of  $ H^1(SK_1(\Zpi))$. This completes the proof. 
\end{proof}

\section{The 2-adic universal homomorphisms}\label{sec:two}
Information about the 2-adic $\kappa$-homomorphisms is provided by \cite[\S 7]{Hambleton:1988} and Milgram-Oliver \cite{Milgram:1990}. We will recall these results below.

\subsection{The $2$-adic $\kappa_2^s$ homomorphisms}

\begin{theorem}[Milgram-Oliver \cite{Milgram:1990}]\label{thm:kappatwo}
Let $\pi$ be a finite $2$-group.
There is a commutative diagram:
$$\xymatrix{ H_2(\pi; \bZ) \ar[d]_{\kappa^s_2}\ar[r]^(0.4){\theta}]& H^1(SK_1(\Zhat\pi))\ar[d]^{\bd}\\
L_0^s(\Zpi)\ar[r]^{\rho_2} & L_0^s(\Zhat\pi)
}$$
where $\rho_2$ is the natural map to the 2-adic $L$-group.
\end{theorem}

\begin{corollary} If $\pi$ is abelian, then $\bar\kappa_2^s(\pi) \colon H_2(\pi;\cy 2) \to L_0^s(\Zhat\pi)$ is zero.\end{corollary}

Oliver showed in \cite[Theorem 3]{Oliver:1980a} that there is a natural isomorphism
$$\theta\colon H_2(\pi;\bZ)/H_2^{\ab}(\pi) \xrightarrow{\cong} SK_1(\Zhat\pi)$$
where $H_2^{\ab}(\pi) \subset H_2(\pi; \bZ)$ denotes the image under induction from abelian subgroups of $\pi$. Furthermore, Oliver proved that the standard involution  induced by $g\mapsto g^{-1}$ is multiplication by $-1$ on $SK_1(\Zhat\pi)$ (see \cite[Theorem 8.6]{Oliver:1988a}), hence 
$H^1(SK_1(\Zhat\pi)) \cong SK_1(\Zhat\pi)\otimes \cy 2$.
\begin{lemma}\label{lem:threetwo} $L'_0(\Zhat\pi) \cong H^1(\Wh'(\Zhat\pi))$ and $L'_2(\Zhat\pi)\cong \cy 2 \oplus H^1(\Wh'(\Zhat\pi))$. 
\end{lemma}

\begin{proof}
We have the Ranicki-Rothenberg sequence
$$ \dots \to  L^h_{n+1} (\Zhat\pi) \to H^n(\Wh'(\Zhat\pi)) \to L'_n(\Zhat\pi) \to L^h_n(\Zhat\pi)\to \dots$$
and the calculations $L^h_n(\Zhat\pi) = \cy 2$, for $n$ even, detected by the discriminant or the Arf invariant, and $L^h_n(\Zhat\pi) =0$ for $n$ odd  (see 
\cite[\S 1.2]{Wall:1976a} and \cite[\S 2, Theorem 3.1]{Hambleton:2000} for details, background on the ``round" $L$-groups, and their relation to the surgery obstruction groups). Hence 
we obtain an isomorphism $H^1(\Wh'(\Zhat\pi)) \to L'_0(\Zhat\pi)$ and a split injection $H^1(\Wh'(\Zhat\pi)) \to L'_0(\Zhat\pi)$,  with cokernel $\cy 2$ detected by the Arf invariant. 

Similarly, we have an exact sequence
$$0 \to \cy 2 \to H^0(\Wh'(\Zhat\pi)) \to  L'_3(\Zhat\pi)\to 0$$
and  $H^0(\Wh'(\Zhat\pi)) \xrightarrow{\cong}  L'_1(\Zhat\pi)$.
\end{proof}

\begin{proof}[The proof of Theorem C] 
By assumption, we have $L'_0(\Zhat\pi)\cong H^1(\Wh'(\Zhat\pi)) =0$ and the boundary map $\bd$ in the sequence
$$ \dots \to L_1'(\Zhat\pi) \to H^1(SK_1(\Zhat\pi)) \xrightarrow{\bd} L_0^s(\Zhat\pi) \to L'_0(\Zhat\pi) \to \dots $$
 is non-zero (and surjective). Since $\bar\kappa_2^s(\pi) = \rho_2\circ \kappa_2^s(\pi)$ and $\bar\kappa_2^h(\pi)$ factors through $\bar\kappa'_2(\pi) = 0$, the conclusions of Theorem C follow directly from Theorem \ref{thm:kappatwo}.
\end{proof}

The next observation about the torsion in $L$-theory will be used in the proof of Theorem  B, Corollary D and Theorem E. Note that the map  $L_*(\Zpi) \to L_*(\bR \pi)$ has finite kernel and cokernel, and  $ L_*(\bR\pi)$ is detected by the multi-signature (see \cite[\S 7]{Wall:1974}, \cite[\S 2.2]{Wall:1976a}).
\begin{lemma}\label{lem:threethree} For $\pi$ a finite $2$-group, the torsion subgroup of $L_{2k}'(\Zpi)$ is detected by the natural map  
$L_{2k}'(\Zpi) \to L_{2k}'(\Zhat\pi)\oplus L_{2k}'(\Zpi^{\ab})$. 
\end{lemma}
\begin{proof} We first consider the case $k$ even. Suppose that $x \in \ker\{L_0'(\Zpi) \to  L_0'(\Zhat\pi)\oplus L_0'(\Zpi^{\ab})$ is a torsion element. 
In  the long exact sequence
$$ \dots \to L_1'(\Zhat\pi) \xrightarrow{\psi_1} L_1'(\Zpi \to \Zhat\pi) \xrightarrow{\bd} L_0'(\Zpi) \to L_0'(\Zhat\pi)\to \dots$$
a torsion element $x \in \ker\{ L_0'(\Zpi) \to L_0'(\Zhat\pi)\}$ is the image $\bd(y) = x$ for some torsion element $y \in L_1'(\Zpi \to \Zhat\pi)$. By Wall
\cite[5.2.2]{Wall:1976a}, the torsion subgroup of the relative $L$-group is a direct sum of copies of $\cy 2$, one for each type $O$ simple factor of the rational group algebra $\bQ\pi$. We can apply the Quadratic Generation Theorem \cite[Theorem 1.B.7 \& Example 1.B.8 (iii)]{Hambleton:1990} to show that the torsion subgroup of $ L_1'(\Zpi \to \Zhat\pi)$ is generated by the images of torsion elements in the relative $L$-groups of basic subquotients of $\pi$. 

Since the linear type $O$ representations are determined by the map $\pi \to \pi^{\ab}$, we only need to consider the (non-abelian) dihedral subquotients. However,  the Wall groups $L_0'(\bZ D(2^{r}))$, $r \geq 3$, are torsion free
 \cite[Theorem 5.2.3]{Wall:1976a}. It follows that $y \in L_1'(\Zpi \to \Zhat\pi)$ lies in the image of $\psi_1$ and hence $\bd(y) = x = 0$.

For $k$ odd, a similar argument works for $L'_2(\Zpi)$ since the  relative group $L_3'(\Zpi \to \Zhat\pi)$ is a sum of copies of $\cy 2$ over representations of type $Sp$, and hence is generated by quaternion subquotients. However, $L_2'(\bZ Q(2^{r})) \cong \cy 2$, for $r \geq 3$, is detected by the trivial group  \cite[Theorem 5.2.3]{Wall:1976a} and the result follows.
\end{proof}

\subsection{The $2$-adic $\kappa'_2$ homomorphisms}

\begin{proposition}[{\cite[7.2]{Hambleton:1988}}]\label{prop:kappatwoprime}
Let $\pi$ be a finite $2$-group.
There is a commutative diagram:
$$
\vcenter{\xymatrix{ H_2(\pi;\cy 2) \ar[d]_{\kappa'_2}\ar[r]^{\beta} & 
H^0(\pi^{\ab}) \ar[r]^(0.4)\delta & H^1(\Wh'(\Zhat\pi)) \ar[d]^\bd\\
L'_0(\Zpi) \ar[rr]^{\rho_2}&& L'_0(\Zhat\pi)
}}$$
\end{proposition}
\begin{remark}
We will lift this diagram to $R\pi$ and use the factorization $\kappa_j^s = \trf\circ \wkappa_j^s$, where
 $\wkappa^s_j \colon H_j (\pi;\cy 2) \to L_j^s(R\pi)$ and $\trf\colon L_j^s(R\pi) \to L_{j+2}^s(\Zpi)$. The  quadratic extension ring $R = \bZ[\varepsilon]$, with $\varepsilon = (1+ \sqrt{5})/2$ has the important properties that $\varepsilon + \bar\varepsilon = 1$ and $\varepsilon\bar\varepsilon = -1$ (see \cite[Theorem 1.16]{Hambleton:1988}).
 \end{remark}
 
The following calculations will be needed (particularly for  $\pi = C_2$ of order two). The details can be found in \cite[3.3, 3.11, 4.5, 4.11]{Hambleton:1988}:
\begin{enumerate}
\item  $L^{\tilde Y}_i(\Rhat\pi)  \cong  L^{\tilde Y}_i(\Rhat) \oplus H^0(\pi^{\ab})$;
\item   $L^{\tilde Y}_i(\Rhat)  = \cy 2, 0, 0, \cy 2$;
\item  $L^{\tilde Y}_3(\Rpi \to \Rhat\pi) = (\cy 2)^{s_2}$, where $s_2$ counts the type $O$ factors in $\bQ\pi$;
\item  $L^{\tilde Y}_1(\Rpi \to \Rhat\pi) = (\cy 2)^{s_0}$, where $s_0$ counts the type $Sp$ factors in $\bQ\pi$;
\item $L^{\tilde Y}_i(R)  = \cy 2, \cy 2, 0, \cy 2 $;
\item  $L^{\tilde Y}_i(R[C_2])  = (\cy 2)^2, (\cy 2)^3,  \cy 2, 0 $.
\end{enumerate}

The round torsion decoration $Y = R^\times \oplus \pi^{\ab} \oplus SK_1(A\pi) \subseteq K_1(A\pi)$, for $A = R$ or $\Rhat$, and $\widetilde Y$ denotes the quotient dividing out $K_1(\bZ) = \{\pm 1\}$. In the proof of Theorem B, we will also need the round \emph{simple} torsion decoration $V = \Rhat^\times \oplus \pi^{\ab}\subseteq K_1(\Rhat\pi)$ and the surgery groups $L^s_*(\Rhat\pi) := L^{\tilde V}_*(\Rhat\pi)$.

\begin{proof}[The proof of Proposition \ref{prop:kappatwoprime}]
The argument will be divided into two parts. 
First we claim that the following diagram
\eqncount
\begin{equation}\label{diag:threefive}
\vcenter{\xymatrix{H_2(\pi;\cy 2) \ar[d]_{\wkappa'_2}\ar[r]^{\beta} & 
H^0(\pi^{\ab}) \\
L^{\tilde Y}_2(\Rpi) \ar[r]^{\rho_2}& L^{\tilde Y}_2(\Rhat\pi) \ar[u]_\tau
}}\end{equation}
 commutes, where $\beta\colon H_2(\pi;\cy 2) \to H^0(\pi^{\ab}) = \{ [g] \in \pi^{\ab}:  [g^2] = 1\}$ is induced by the Bockstein and $\tau$ by the discriminant.
   By the 2-adic Detection Theorem \cite[Theorem 3.4]{Hambleton:1988} and naturality, it is enough to show that the diagram commutes for $\pi = C_2$ of order two. In that case,  $L^{\tilde Y}_2(\Rpi) \cong \cy 2$  maps isomorphically onto 
  $L^{\tilde Y}_2(\Rhat\pi) = \cy 2$, since $\tilde\psi_3 \colon L^{\tilde Y}_3(\Rhat\pi) \to L^{\tilde Y}_3(\Rpi \to \Rhat\pi)$ is an isomorphism  (see \cite[Theorem 4.10]{Hambleton:1988}) and $\pi = C_2$ has two type $O$ representations.  The map $\beta$ is an isomorphism, and the discriminant $\tau\colon L^{\tilde Y}_2(\Rhat\pi)  \to H^0(\pi^{\ab})$ is also an isomorphism. Therefore the diagram commutes for all finite $2$-groups.
  
Next we consider the following diagram (for any finite $2$-group):
\eqncount
\begin{equation}\label{diag:threesix}
\vcenter{\xymatrix{  
H^0(\pi^{\ab})\ar@/^1pc/[dr]^{\tilde\delta}_\cong& \\
 L^{\tilde Y}_2(\Rhat\pi) \ar[u]_\tau^\cong \ar[d]^{\trf}& H^1(\Wh^{\tilde Y}(\Rhat \pi)) \ar[l]^(0.6)\bd_(0.6)\cong \ar[d]^{\trf}\\
 L'_0(\Zhat\pi) & H^1(\Wh'(\Zhat \pi)) \ar[l]^(0.6)\bd
}}
\end{equation}
where the map $\tilde \delta$ is the coboundary induced in Tate cohomology by the upper sequence in the diagram (see \cite[Theorem 2]{Oliver:1980a}):

\eqncount
\begin{equation}\label{diag:threeseven}
\vcenter{\xymatrix{ 0 \ar[r]& \Wh^{\tilde Y}(\Rhat \pi)\ar[r]\ar[d]& \overline{I(\Rhat\pi)} \ar[r]\ar[d]& \pi^{\ab}\ar[r]\ar@{=}[d]& 0\\
0 \ar[r]& \Wh'(\Zhat \pi)\ar[r]& \overline{I(\Zhat\pi)} \ar[r]& \pi^{\ab}\ar[r]& 0}}
\end{equation}
Since $\varepsilon + \bar\varepsilon = 1$, we have $H^*(\overline{I(\Rhat\pi)}) = 0$ and hence $\tilde\delta$ is an isomorphism. By naturality, 
$\delta = \trf\circ \tilde \delta \colon H^0(\pi^{\ab})\to  H^1(\Wh'(\Zhat \pi)) $. By combining diagrams \eqref{diag:threefive} amd \eqref{diag:threesix} above, the proof is complete.
\end{proof}

  \begin{proof}[The proof of Theorem B]
 The parts (ii) and (iii) follow directly from the properties of the lifted $\wkappas_*$ listed in \cite[Theorem 6.8]{Hambleton:1988}, since $L^{\tilde Y}_*(\Rhat\pi)$ is detected by $\pi \to \pi^{\ab}$ (by \cite[Theorem 3.4]{Hambleton:1988}) and 
 $L^s_*(\Rhat\pi) = L^{\tilde Y}_*(\Rhat\pi)$ for $\pi$ abelian (since the transfer $\trf\colon SK_1(\Rhat\pi) \to SK_1(\Zhat\pi)$ is an isomorphism by \cite[8.7]{Oliver:1988a} ). For part (i) note that $\kappa^s_1$ is generated/detected by cyclic groups, and \cite[3.3]{Wall:1976a} shows that for cyclic $2$-groups $L'_3(\Zpi) \cong L'_3(\Zhat\pi)$.
 The last statement follows from Lemma \ref{lem:threethree}. Some examples of groups $\pi$ satisfying all the required conditions are given in Corollary \ref{cor:Ethree}.
  \end{proof}
  
  \begin{proof}[The proof of Corollary D] The first step is to find a group $\pi$ satisfying the conditions listed in Theorem C. This will show that $\kappa^s_2(\pi) \neq 0$ and that $\bar\kappa_2^h(\pi) = 0$. By Lemma \ref{lem:threethree},  the homomorphism $\kappa'_2(\pi)$ is detected by the map $L'_0(\Zpi) \to L'_0(\Zpi^{\ab})$. However, by naturality the image of $\kappa'_2(\pi)$ in $L'_0(\Zpi^{\ab})$ equals the image of  $\kappa'_2(\pi^{\ab})$. By Theorem A, we can conclude that $\kappa'_2(\pi^{\ab}) = \kappa^h_2(\pi^{\ab}) = 0$ whenever $\pi^{\ab}$ is an elementary abelian group, and the result follows. See Examples \ref{ex:corDgroups} for a list of groups of order 128 satisfying Corollary D. 
  
  \end{proof}  
\subsection{The $\kappa'_4$ homomorphisms}
In the following diagram, $s_*\colon H_4(\pi;\cy 2) \to H_2(\pi;\cy 2)$ denotes the Kronecker dual of the map $Sq^2$ in mod 2 cohomology.
\begin{proposition}[{\cite[7.3]{Hambleton:1988}}]\label{prop:kappafourprime}
Let $\pi$ be a finite $2$-group.
There is a commutative diagram:
\eqncount
\begin{equation}
\vcenter{\xymatrix{ H_4(\pi;\cy 2) \ar[dr]_{\kappa'_4} \ar[r]^{s_*} &H_2(\pi;\cy 2) \ar[r]^{\beta} & 
H^0(\pi^{\ab}) \ar[r]^(0.4)\delta & H^1(\Wh'(\Zhat\pi) \ar[dl]^\bd\\
& L'_2(\Zpi) \ar[r]^{\rho_2}& L'_2(\Zhat\pi)&
}}\end{equation}
\end{proposition}
\begin{proof} The argument is similar to that for the $\kappa'_2$ diagram. We first lift to a diagram over $R$, where it is enough to check commutativity for $\pi = C_2$.
\eqncount
\begin{equation}\label{diag:threenine}
\vcenter{\xymatrix{H_4(\pi; \cy 2) \ar[r]^{s_*} \ar[dr]_{\kappa'_4} &H_2(\pi;\cy 2) \ar[r]^{\beta} & 
H^0(\pi^{\ab}) \\
& L^{\tilde Y}_0(\Rpi) \ar[r]^{\rho_2}& L^{\tilde Y}_0(\Rhat\pi) \ar[u]_\tau
}}\end{equation}
In this diagram for $\pi = C_2$, the map  $L^{\tilde Y}_0(\Rpi) \cong (\cy 2)^2$ to 
  $L^{\tilde Y}_0(\Rhat\pi) = (\cy 2)^2$ is an isomorphism (since $\bQ\pi$ has no type $Sp$ factors),  with one summand detected by the the trivial group $L^{\tilde Y}_0(R) = \cy 2$. Note that the discriminant map $\tau\colon L^{\tilde Y}_0(\Rhat\pi)  \to H^0(\pi^{\ab})$ is induced by the projection $\pi \to \pi^{\ab}$ (just the identity for $\pi = C_2$). Since the image of $\kappa'_4(\pi)$ is zero for the trivial group, and both maps $s_*$ and $\beta$ are isomorphisms, the diagram commutes.
  
  The remaining part of the proof uses a similar diagram to \eqref{diag:threesix}:
\eqncount
\begin{equation}\label{diag:three10}
\vcenter{\xymatrix{  
H^0(\pi^{\ab})\ar@/^1pc/[dr]^{\tilde\delta}_\cong& \\
 L^{\tilde Y}_0(\Rhat\pi) \ar@{->>}[u]_\tau \ar[d]^{\trf}& H^1(\Wh^{\tilde Y}(\Rhat \pi)) \ar@{->>}[l]^(0.6)\bd \ar[d]^{\trf}\\
 L'_2(\Zhat\pi) & H^1(\Wh'(\Zhat \pi)) \ar[l]^(0.6)\bd
}}
\end{equation}
   and diagram \eqref{diag:threeseven} together with naturality under $\trf$ to complete the argument. 
\end{proof}

\section{Non-vanishing examples for $\kappa_2$}\label{sec:three}
In this section we establish some conditions for verifying the statements of Theorem C and Corollary D. We should point out that when $H^1(\Wh'(\Zhat\pi)) = 0$, 
we get $\bar\kappa'_2(\pi) = 0$ and so  the image of $\bar\kappa^s_2(\pi)$ lies in the image of $\bd\colon H^1(SK_1(\Zhat\pi)) \to L_0^s(\Zhat\pi)$. Since $\theta\colon H_2(\pi;\bZ) \to   H^1(SK_1(\Zhat\pi))$ is surjective, we get $\bar\kappa^s_2(\pi) \neq 0$ if and only if $\Image \bd \neq 0$.

\subsection{Group theoretical conditions to construct examples}

Here we define a set $\Lambda(\alpha)$ and two maps $u_{\alpha}$ and
$\delta^{\alpha}$ for any central extension $$ \sigma { \rightarrowtail
} \wpi  \overset{ \alpha }{ \twoheadrightarrow } \pi   $$
where $\wpi $ is a $2$-group. Given such an extension, take
any surjective group homomorphism $f\colon F\rightarrow \wpi$ where
$F$ is a free group. Let $R$ be the kernel of the composition $\alpha
\circ f$. Since $f(R)\subseteq \sigma \subseteq Z(\wpi)$, we have
$ f([F,R])=1$ and
therefore we get a
well-defined homomorphism from $F/[F,R]$ to $\wpi$. Let
$u_{\alpha}$ denote the restriction of this map to  $\pi\wedge \pi$
where $\pi\wedge \pi=[F,F]/[F,R]$ (see proof of Statement 11.4.16 in
\cite{Robinson:1993} ). The map $$u_{\alpha}\colon \pi \wedge \pi
\rightarrow \wpi$$ restricts to a group homomorphism
$$\delta^{\alpha}\colon H_2(\pi;\bZ)=\frac{R\cap
[F,F]}{[F,R]}\rightarrow \sigma \, .$$
Let $q_1,q_2$ be in $\pi$ and $\widehat{q}_1,\widehat{q}_2$ be their
lifts in $F$. Then  $[\widehat{q}_1,\widehat{q}_2]$ is in $[F,F]$ and
defines an element $q_1\wedge q_2$ in $\pi\wedge \pi$ that only depends
on $q_1,q_2$. If we further assume that $q_1q_2=q_2q_1$ then we have
$q_1\wedge q_2$ in $H_2(\pi;\bZ)$. Now we define a subset
$\Lambda(\alpha)$ of $H_2(\pi;\bZ)$ as follows:
$$\Lambda(\alpha)=\left\{ \, q_1\wedge q_2 \,\left|\,  q_1,q_2\in \pi
\textup{ and }q_1q_2=q_2q_1 \right. \right\}.$$
We have
$$SK_1(\Zhat\wpi)\cong \frac{\ker (\delta
^{\alpha})}{\left\langle \ker (\delta ^{\alpha}) \cap \Lambda(\alpha)
\right\rangle}$$
by Part (i) of Lemma 22 in \cite{Oliver:1980a}. The following is a special case of a result of Oliver.

\begin{theorem}[{\cite[Proposition 16]{Oliver:1980a}}]\label{thm:t1}
Let $\wpi$ be a $2$-group and $ \sigma { \rightarrowtail  }
\wpi  \overset{ \alpha }{ \twoheadrightarrow } \pi   $ a
central group extension. Assume that $1\neq \sigma\subseteq
[\wpi,\wpi]$ and $\sigma\cap \Omega (\alpha) = 1$
where $$\Omega (\alpha)  = \left\{
\,[\widetilde{g}_1,\widetilde{g}_2]\,\, \left| \,\,
\widetilde{g}_1,\widetilde{g}_2\in \wpi \textup{ and }
\alpha([\widetilde{g}_1,\widetilde{g}_2])=1\, \right. \right\}.$$ Then
$SK_1(\Zhat\wpi)\overset{ \alpha_* }{ \rightarrow
}SK_1(\Zhat\pi)$ is not surjective. In particular this means
$SK_1(\Zhat\pi)\neq 0$.
\end{theorem}
\begin{proof}
Let $\iota \colon \pi \to \pi$ denote the identity map. Then $ 1 {
\rightarrowtail  } \pi \overset{ \iota }{ \twoheadrightarrow } \pi   $
is also a central group extension. Take any surjective group
homomorphism $f\colon F\rightarrow \wpi$ where $F$ is a free
group. Then the map $\alpha \circ f\colon F\rightarrow \pi$ is also
surjective. Hence, by definition, we have
$\Lambda(\iota)=\Lambda(\alpha)$ as subsets of $\pi\wedge\pi$. We also
have  $\alpha \circ u_{\alpha}=u_{\iota}$ and hence $\alpha \circ \delta
^{\alpha}=\delta ^{\iota}$. Let $q_1$, $q_2$ be two elements in $\pi $
such that $q_1q_2=q_2q_1$ and let $\widehat{q}_1,\widehat{q}_2$ be their
lifts in $F$. Then we have $u_{\alpha}(q_1\wedge
q_2)=u_{\alpha}([\widehat{q}_1,\widehat{q}_2])=[u_{\alpha}(\widehat{q}_1),u_{\alpha}(\widehat{q}_2)]\in
\Omega(\alpha ). $ Hence $\Lambda(\alpha )\subseteq
u_{\alpha}^{-1}(\Omega (\alpha))$. Moreover, $\sigma\subset
[\wpi,\wpi]$, Hence $\sigma$ is in the image of
$\delta^{\alpha}$. We also have  $\sigma \neq 1$ and therefore
$\ker(\delta^{\alpha})\subsetneqq \ker(\delta^{\iota}) $. This means
$$\Lambda(\iota )\cap \ker(\delta^{\iota })\subseteq
u_{\alpha}^{-1}(\Omega (\alpha))\cap u_{\widetilde{\alpha}}^{-1}(\sigma
)\subseteq
u_{\widetilde{\alpha}}^{-1}(\Omega _p\cap \sigma)=
u_{\widetilde{\alpha}}^{-1}(1)=\ker(\delta^{\alpha})\subsetneqq
\ker(\delta^{\iota}) $$
Hence the map $\alpha_*$ given by the composition
$$SK_1(\Zhat\wpi)\cong\frac{\ker (\delta
^{\alpha})}{\left\langle \ker (\delta ^{\alpha}) \cap \Lambda(\alpha)
\right\rangle}\rightarrow \frac{\ker (\delta ^{\iota})}{\left\langle
\ker (\delta ^{\alpha}) \cap \Lambda(\alpha) \right\rangle}\rightarrow
\frac{\ker (\delta ^{\iota})}{\left\langle \ker (\delta ^{\iota}) \cap
\Lambda(\iota) \right\rangle} \cong SK_1(\Zhat\pi) $$
is not surjective.
\end{proof}
\begin{remark}\label{rem:r1} 
      The ``converse'' of the last part of Theorem \ref{thm:t1} is true:  if $SK_1(\Zhat\pi)\neq 0$, then  there exists a suitable central extension of $\pi$ so that the $SK_1$ is not surjective.  To check this, 
 let  $ H_2(\pi) { \rightarrowtail}
      \bar{\pi}  \overset{ \alpha }{ \twoheadrightarrow } \pi   $ be a Schur
cover of $\pi$. Notice that $SK_1(\Zhat\bar{\pi})= 0$. Let $\sigma _0  =
H_2(\pi) \cap [\bar{\pi},\bar{\pi}]$ and
      $\sigma _1 = \langle\, c \in  H_2(\pi) \,|\, c \text{ is a commutator
in } \bar{\pi} \,\rangle$. Notice we have $\sigma _0 =  H_2(\pi)$ and
by Proposition 16 in \cite{Oliver:1980a},  $\sigma_0/\sigma _1$ is
nontrivial.  Since $\sigma_0$ is an abelian group,  there exists an
element $z$ in $\sigma _0 $ and there exists a subgroup $T$ in
$\sigma_0$ such that $z$ is not in $T$,  $\sigma _1$ is a subgroup of
$T$ and $\sigma_0/T$ is $\cy 2$. Now define $\sigma = H_2(\pi)/T=\cy 2$
and $\wpi=\bar{\pi}/T$. Then $ \sigma { \rightarrowtail  }
      \wpi  \overset{ \alpha }{ \twoheadrightarrow } \pi   $
satisfies the hypothesis of the Theorem \ref{thm:t1}.
\end{remark}
\begin{theorem}\label{thm:t2}
Let $\wpi$ be a $2$-group and $ \sigma { \rightarrowtail  }
\wpi  \overset{ \alpha }{ \twoheadrightarrow } \pi   $ a
central group extension such that $1\neq \sigma\subseteq
[\wpi,\wpi]$.  Assume that every element in $\pi $
which is conjugate to its inverse lifts to some element in
$\wpi $ with the same property and
$H^1(SK_1(\Zhat\wpi))\overset{ \alpha_* }{ \rightarrow
}H^1(SK_1(\Zhat\pi))$ is not surjective. Then the map
$H^1(SK_1(\Zhat\pi)) \to L_0^s(\Zhat\pi)$ is not trivial.
\end{theorem}
\begin{proof}
In the proof of Proposition 3.1 in \cite{Milgram:1990}, it is shown that under the assumption above
any element in the kernel of the map $H^1(SK_1(\Zhat\pi)) \to
L_0^s(\Zhat\pi)$ is in the image of the map
$H^1(SK_1(\Zhat\wpi))\overset{ \alpha_* }{ \rightarrow
}H^1(SK_1(\Zhat\pi))$. Hence the map $H^1(SK_1(\Zhat\pi)) \to
L_0^s(\Zhat\pi)$ is non-trivial when $\alpha_*$ is not surjective.
\end{proof}

\subsection{A non-vanishing $\kappa^s _2$ example}
\label{subsec:k2nonzero} All through this subsection, let $\wpi$ be  the $8177^{th}$ group of order $256$ and
$\pi$ is the $1377^{th}$ group of order $128$ in the Small Groups
Library. 
The group $\widetilde \pi$ of order
$256$ given by the following presentation
\begin{equation*}
    \wpi= \left\langle
    x_1,x_2,x_3,x_4,x_5,x_6,x_7,x_8 \vv r\in R \right\rangle
\end{equation*}
where $R$ contains the relations
$x_1^2 = x_5x_6x_8$,
$x_2^2 = x_5x_6$,
$x_3^2 = x_5$,
$x_4^2 = x_8$,
$x_i^2 = 1$ for $i\in \{5,6,7,8\}$,
$[x_1,x_2] = x_5$,
$[x_1,x_3] = x_6$,
$[x_1,x_4]= x_8$,
$[x_2,x_3]= x_7$,
$[x_2,x_4]=1$, and
$[x_i,x_j]=1$ when $i,j$ are both in $\{3,4,5,6,7,8\}$ or $i$ in
$\{1,2\}$ and $j$ in $\{5,6,7,8\}$. Then the center $Z(\wpi)$ of
$\wpi$  is equal to the commutator group
$[\wpi,\wpi]$ and generated by $x_5,x_6,x_7,x_8$.
Define $$\sigma = \left\langle x_7x_8 \right\rangle\cong \cy 2 \textup{
    and }\pi=\wpi/\sigma$$
Let $\alpha $ be the natural quotient from $\wpi$ to $\pi$:
Then we have a central extension $$ \sigma { \rightarrowtail  }
\wpi  \overset{ \alpha }{ \twoheadrightarrow } \pi   $$ such
that $\sigma\subseteq [\wpi,\wpi]$.  Let $\Omega
(\alpha)$ be the set
$$\Omega (\alpha)  = \left\{  \,[\widetilde{g}_1,\widetilde{g}_2]\,\,
\left| \,\, \widetilde{g}_1,\widetilde{g}_2\in \wpi \textup{
    and } \alpha([\widetilde{g}_1,\widetilde{g}_2])=1\, \right. \right\}$$
Then the following lemma shows that $\Omega(\alpha )$ does not contain
$x_7x_8$.
\begin{lemma}
    $\sigma \cap \Omega(\alpha )=1$
\end{lemma}
\begin{proof}
    Let $n_1,m_1,n_2,m_2,n_3,m_3,n_4,m_4\in \{0,1,2,3\}$ Then we have
    $$\left[\, x_1^{n_1}x_2^{n_2}x_3^{n_3}x_4^{n_4}\, ,\,
    x_1^{m_1}x_2^{m_2}x_3^{m_3}x_4^{m_4}\,
    \right]=x_5^{q_5}x_6^{q_6}x_7^{q_7}x_8^{q_8}$$
    where
    $$ \begin{aligned}
        q_5  & \equiv  n_1 m_2 + n_2 m_1 & \pmod 2, \\
        q_6  & \equiv  n_1 m_3 + n_3 m_1 & \pmod 2, \\
        q_7  & \equiv  n_2 m_3 + n_3 m_2 & \pmod 2, \\        
        q_8  & \equiv  n_1 m_4 + n_4 m_1 & \pmod 2.
    \end{aligned} $$
    We just need to show that there exists no  $n_1,m_1,n_2,m_2,n_3,m_3,n_4,m_4\in \{0,1,2,3\}$ such that $q_5\equiv 0 \Mod 4$, $q_6\equiv 0 \pmod 2$, $q_7\equiv 1 \pmod 2$, and
    $q_8\equiv  1 \pmod 2$. Suppose otherwise. Assume $n_1$ is even. Then $m_1$ is odd since $q_8$ is odd. So $n_2$ is even since $q_5$ is even. Therefore $n_3$ is odd, since $q_7$ is odd. This is a contradiction because $q_6$ is even. Hence $n_1$ is odd. Similarly $m_1$ is odd. Then $m_2 \equiv n_2 \pmod 2$ because $q_5$ is even. We also have $m_3 \equiv n_3 \pmod 2$ because $q_6$ is even. Then $q_7 \equiv 2n_2m_3 \pmod 2$. This is a contradiction. 
\end{proof}
Therefore by Theorem \ref{thm:t1}, we know that the map
$SK_1(\Zhat\wpi)\overset{ \alpha_* }{ \rightarrow
}SK_1(\Zhat\pi)$ is not surjective.
\begin{lemma}\label{lem:fourfour}
    $H^1(SK_1(\Zhat\wpi))\overset{ \alpha_* }{ \rightarrow
    }H^1(SK_1(\Zhat\pi))$ is not surjective.
\end{lemma}
\begin{proof}
    Considering the central extension $ Z(\pi) { \rightarrowtail  } \pi
    \overset{ p }{ \twoheadrightarrow }  \pi/Z(\pi)  $. We know that
    $SK_1(\Zhat\pi)$ is a quotient of $\ker (\delta^{p})$. Since $\ker
    (\delta^{p})\subseteq H_2(\pi/Z(\pi);\bZ)\cong (\cy 2)^4$, we know that
    $SK_1(\Zhat\pi)$ has exponent $2$. Hence
    $H^1(SK_1(\Zhat\wpi))\overset{ \alpha_* }{ \rightarrow
    }H^1(SK_1(\Zhat\pi))$ is not surjective.
\end{proof}
Hence by the following lemma we know that the central extension $\sigma
{ \rightarrowtail  } \wpi  \overset{ \alpha }{
    \twoheadrightarrow } \pi $ satisfies all the hypothesis in Theorem
\ref{thm:t2}.
\begin{lemma}
    All elements in $\pi $ that are conjugate to their inverses have lifts
    in $\wpi$  that have the same property.
\end{lemma}
\begin{proof}
    Suppose $ghg^{-1}=x_7x_8h^{-1}$ for some $g$, $h$ in $\wpi$. Then
    $[g,h]h^2=x_7x_8$. To show that this is impossible, it is enough to consider the case where $g= x_1^{n_1}x_2^{n_2}x_3^{n_3}x_4^{n_4}$ and $h=x_1^{m_1}x_2^{m_2}x_3^{m_3}x_4^{m_4}$    for some $n_1,m_1,n_2,m_2,n_3,m_3,n_4,m_4\in \{0,1,2,3\}$. In this case, we have
    $$[g,h]h^2=\left[\, x_1^{n_1}x_2^{n_2}x_3^{n_3}x_4^{n_4}\, ,\,
    x_1^{m_1}x_2^{m_2}x_3^{m_3}x_4^{m_4}\,
    \right](x_1^{m_1}x_2^{m_2}x_3^{m_3}x_4^{m_4})^2=x_5^{q_5}x_6^{q_6}x_7^{q_7}x_8^{q_8}$$
    where
    $$ \begin{aligned}
        q_5  & \equiv  n_1 m_2 + n_2 m_1 + m_1m_2+m_1+m_2+m_3 & \pmod 2, \\
        q_6  & \equiv  n_1 m_3 + n_3 m_1 + m_1m_3+m_1+m_2 & \pmod 2, \\
        q_7  & \equiv  n_2 m_3 + n_3 m_2 + m_2m_3 & \pmod 2, \\        
        q_8  & \equiv  n_1 m_4 + n_4 m_1 + m_1m_4+m_1+m_4  & \pmod 2.
    \end{aligned} $$
    Hence it is enough to show that there exists no  $n_1,m_1,n_2,m_2,n_3,m_3,n_4,m_4\in \{0,1,2,3\}$ such that $q_5\equiv 0 \Mod 4$, $q_6\equiv 0 \pmod 2$, $q_7\equiv 1 \pmod 2$, and
    $q_8\equiv  1 \pmod 2$. Suppose otherwise. Assume $m_1$ is even. Then $m_4$ is odd and $n_1$ is even odd since $q_8$ is odd. So $m_2$ is even since $q_6$ is even.Therefore $m_3$ is even since $q_5$ is even. This is a contradiction to $q_7$ being odd. Hence $m_1$ is odd. Assume $n_1$ is even. Then $n_2$, $m_3$  must have distinct parity since $q_5$ is even. So $n_2m_3$ is even. This means $m_2$ and $n_3+m_3$ are odd. This is a contradiction since $q_6$ is even. Hence $n_1$ is odd. Then $n_3$ and $m_2$ have distinct parity since $q_6$ is even. So $n_3m_2$ is even. Therefore $n_2+m_2$ and $m_3$ are odd since $q_7$ is odd. This contradicts $q_5$ being even. Hence we are done. 
\end{proof}
Hence by Theorem \ref{thm:t2}
the map $H^1(SK_1(\Zhat\pi)) \to L_0^s(\Zhat\pi)$ is not trivial. Now to
finish the example we prove the following lemma.
\begin{lemma}\label{lem:foursix}
    $H^1(\Wh'(\Zhat \pi))=0$ 
\end{lemma}
\begin{proof}
    Note that
    \eqncount
    \begin{equation}\label{eq:fourseven}
    H^1(\Wh'(\Zhat \pi))\cong \frac{\langle g\in \pi_{\textup{\ab} }  \vv
        g^2=1\rangle }{\langle  g \vv g\sim g^{-1}\rangle }\cong \frac{\langle
        [x_1],[x_2],[x_3],[x_4]\rangle }{\langle [x_1],[x_1x_2],[x_3],[x_4]\rangle }\cong 0
        \end{equation}
    because  we have the relations
    \begin{enumerate}
    \item  $(x_2x_3x_4)x_1(x_2x_3x_4)^{-1}=x_1^{-1}$
    \item  $(x_2x_4)(x_1x_2)(x_2x_4)^{-1}=(x_1x_2)^{-1}$ 
    \item $(x_3x_4)(x_1x_2x_3)(x_3x_4)^{-1}=(x_1x_2x_3)^{-1}$, and
    \item $x_1x_4x_1^{-1}=x_4^{-1}$ \hfill \qedhere
    \end{enumerate}
 
\end{proof}

\begin{corollary}\label{cor:Ethree}  The group $\pi  = SG[128,1377]$ has the property that
$\kappa_2^s(\pi) \neq 0$ and $\bar\kappa'_2(\pi) = 0$. 
Moreover, $\pi^{\ab}$ is elementary abelian and $\kappa^h_2(\pi)  = 0$.
\end{corollary}

\begin{proof} The first property follows from Lemma \ref{lem:fourfour},
Lemma \ref{lem:foursix} and Theorem \ref{thm:t2}, which establish the
conditions of Theorem C. 
By inspection, we see that $\pi^{\ab}$ is elementary abelian, and hence $\kappa'_2(\pi^{\ab}) = 0$. Since $H^1(\Wh'(\Zhat \pi))=0$, we have $\bar\kappa'_2(\pi) = 0$ and Theorem B implies that $\kappa'_2(\pi) = 0$. 
Therefore $\kappa^h_2(\pi) = 0$ and $\ker \kappa_2^s \subsetneq \ker\kappa_2^h$.
\end{proof}

\begin{remark}\label{rem:fournine}
Let  $\pi$ be a $2$-group of order less than or equal to $64$  and $\sigma { \rightarrowtail  }
\wpi  \overset{ \alpha }{ \twoheadrightarrow } \pi   $  a central group extension which
satisfies the hypothesis of Theorem \ref{thm:t1}. Then $SK_1(\Zhat\pi)$ is non-zero, and  by Remark \ref{rem:r1}, without loss of generality we can assume that $\sigma = \cy 2$ in this central extention.  Then by a GAP computation (see Appendix \ref{sec:apptwo}), we show that  $\pi$ is of order $64$ with group number $149$, $150$,  $151$, $170$, $171$, $172$,  $177$, or $182$. Moreover, $H^1(\Wh'(\Zhat\pi))=0$ for the groups in this list. 

One can ask if these groups can fit into a central extension that satisfies both the hypotheses of Theorem \ref{thm:t1} and Theorem \ref{thm:t2}.  
\end{remark}

\begin{remark}\label{rem:fourtwelve}  Let $\pi $ be the $i$th group of order $64$ where $i$ is one of numbers in the previous list \eqref{rem:fournine}. Then by a GAP computation (see Appendix \ref{sec:apptwo}, Listing 6) we show that $SK_1(\Zhat\pi)\cong \cy 2$.  Then there does not exist
 any central group extension $\sigma { \rightarrowtail  } \wpi  \overset{ \alpha }{ \twoheadrightarrow } \pi   $   which satisfies the hypothesis of Theorem \ref{thm:t1}, and those of Theorem \ref{thm:t2}, where $\sigma\cong \cy 2$.

 We expect that  $0 = \bar\kappa_2^s(\pi)\colon H_2(\pi;\cy 2) \to L^s_0(\Zhat\pi)$ for all the groups in the list \eqref{rem:fournine}. If this holds, then there will not exist any groups $\pi$ of order $\leq 64$ for which $\bar\kappa^s_2(\pi) \neq 0$. Note that  all these groups  have $H^1(\Wh'(\Zhat\pi))=0$, hence $\bar\kappa'_2(\pi)=0$.
\end{remark}

\begin{example}\label{ex:thmEgroups}
  A complete list of extensions among groups of order 128  that satisfies the hypothesis of
Theorem \ref{thm:t1} and Theorem \ref{thm:t2} can be given as follows: 
\begin{enumerate}
\item[] $ (287 , 1327 ) $,  $ (288 , 1328 ) $,  $ (290 , 1330 ) $,  $
(693 , 3996 ) $,  $ (692 , 3998 ) $,  $ (704 , 4010 ) $, \\
$ (703 , 4012 ) $,  $ (668 , 4054 ) $,  $ (667 , 4055 ) $,  $ (670 ,
4056 ) $, $ (568 , 4255 ) $, $ (579 , 4312 ) $, \\
 $ (676 , 4316 ) $,  $ (725 , 4372 ) $,  $ (1375 , 7736 ) $,  $ (1376 ,
8129 ) $,   $ (1377 , 8177 ) $, \\
  $ (1547 , 9240 ) $,  $ (1549 , 9242 ) $,  $ (1576 , 10060 ) $
\end{enumerate}
where the pair $(i,j)$ being in this list means the $j$th group of order
$256$ has a quotient that is isomorphic to the $i$th group of order
$128$ and considered as a $\cy 2$ extension it satisfies the conditions
in Theorem \ref{thm:t1}  and Theorem \ref{thm:t2}.  All these pairs also
satisfy the result of Lemma \ref{lem:foursix} except  $ (568 , 4255 ) $.
Hence from the above list we can obtain a complete list of the minimal
 examples (all of order 128) that satisfy the conditions of Theorem C. 
 For all the groups in the above list, $SK_1(\Zhat\pi) \cong H_2(\pi;\bZ)/H^{\ab}_2(\pi) \neq 1$,
  by a GAP calculation (see Appendix \ref{sec:apptwo}),  and all except for $ (1547 , 9240 ) $,  $ (1549 , 9242 ) $,  $ (1576 , 10060 ) $
  have $SK_1(\Zhat\pi) = \cy 2$. The last three have $SK_1(\Zhat\pi)\cong \cy 2 \oplus \cy 2$.
\end{example}

\begin{example}\label{ex:corDgroups} The groups of order 128 in the list \eqref{ex:thmEgroups} with elementary abelianization are the following: 

$ (1375 , 7736 ) $,  $ (1376 ,
8129 ) $,   $ (1377 , 8177 ) $, 
  $ (1547 , 9240 ) $,  $ (1549 , 9242 ) $,  $ (1576 , 10060 ) $.

\smallskip\noindent
These groups all satisfy Corollary D.
\end{example}
%%%%%%%%%%%%%%%%%%%%%%%%%%%%%%%%%%%%%%%%%%%%%%%%%

\begin{remark}
      Let
              $\pi$ be the the Sylow $2$-subgroup of $U_3(4)$, which is $SG[64,245]$ in
the Small Groups Library.
              Then  $H^1(\Wh'(\Zhat\pi))=0$. Since $\pi$ has no $\cy 2$ extension
that satisfy the hypothesis of Theorem \ref{thm:t1},   we know that $H^1(SK_1(\Zhat\pi))=0$ by Remark
\ref{rem:r1}.  Hence we cannot use
our method for $\pi$ to decide if $\kappa^s_2(\pi) \neq0$. It is therefore important to study $SK_1(\Zpi)$ in more detail.      \end{remark}

      \begin{question} At present we do not have enough information about  the subgroup $Cl_1 (\Zpi) \subseteq SK_1(\Zpi)$ to make effective calculations of the maps in the Ranicki-Rothenberg exact sequences
$$ \dots \to H^1(SK_1(\Zpi)) \to L_0^s(\Zpi) \to L'_0(\Zpi) \to H^0(SK_1(\Zpi)) \to \dots$$
or in the exact sequence in Tate cohomology
$$ \dots \to H^0(SK_1(\Zhat\pi)) \to H^1(Cl_1(\Zpi)) \to H^1(SK_1(\Zpi))\to H^1(SK_1(\Zhat\pi)) \to \dots$$
This remains a challenge for the future (see \cite[Theorem 9.6]{Oliver:1988a}). Here are some sample questions:
      \begin{enumerate}
      \setlength{\itemsep}{0pt plus 8pt}  
      \item      Let $\pi$ be the the Sylow $2$-subgroup of $U_3(4)$. Is the map
$H^1(SK_1(\Zpi)) \to L_2^s(\Zpi)$ non-trivial ? For this group $\bar\kappa^s_2(\pi) = 0$ since $SK_1(\Zhat\pi) = 0$.
\item Let $\pi$ be one of the groups in the list \eqref{ex:thmEgroups}, but not in \eqref{ex:corDgroups}.
 Is it true that $\kappa'_2(\pi) \neq 0$ ? Does $\ker\kappa_2^s(\pi) = \ker\kappa_2^h(\pi)$ hold for these groups ?
\item More generally, do there exist examples where $\bar\kappa^s_2(\pi) = 0$, but $\kappa^s_2(\pi) \neq 0 $ ?
\item Do there exist examples where $\ker\kappa_2^s(\pi) \subsetneq \ker\kappa_2^{Cl_1}(\pi)$ ?
 
 \end{enumerate} 
\end{question}

\section{The proof of Theorem E}\label{sec:five}

We reduce the original ``oozing conjecture" for surgery up to homotopy equivalence to an explicit calculation in group homology.  As an application, we provide a counterexample to the conjecture (see Section \ref{subsec:fivethree}), but we are far from a systematic understanding about how to generate such examples. Are there infinitely many finite $2$-groups with non-trivial $\kappa'_4$ ?
For any finite $2$-group $\pi$, we introduced the notation
$$\lambda_4(\pi) := \delta\circ \beta\circ s_* \colon  H_4(\pi;\cy 2) \to H^1(\Wh'(\Zhat\pi))$$
for the composite, where 
$s_*\colon H_4(\pi;\cy 2) \to H_2(\pi; \cy 2)$ denote the dual homomorphism to $Sq^2 \colon H^2(\pi;\cy 2) \to H^4(\pi;\cy 2)$, and $\beta\colon H_2(\pi;\cy 2) \to H^0(\pi^{\ab})$ induced by the Bockstein. In addition, the short exact sequence
$$0 \to \Wh'(\Zpi) \to \overline{I(\Zhat\pi)} \to \pi^{\ab} \to 0$$
induces a (surjective) coboundary map $\delta\colon H^0(\pi^{\ab}) \to  H^1(\Wh'(\Zhat\pi))$. For the notation see \cite[p.~163]{Oliver:1988a}. 

\begin{proof}[The proof of Theorem E]
By Proposition \ref{prop:kappafourprime}, we have a commutative diagram:
$$\xymatrix@C+7pt{
H_4(\pi;\cy 2) \ar[r]^(0.4){\lambda_4(\pi)} \ar[d]_{\kappa'_4(\pi)} & H^1(\Wh'(\Zhat\pi))\ar[d]^{\bd}\\
L'_2(\Zpi) \ar[r]^{\rho_2} & L'_2(\Zhat\pi)
}$$
in which the boundary map $\bd$ is injective. 
By Lemma \ref{lem:threethree}, the torsion subgroup of $L'_2(\Zpi)$ is detected by the natural map
$$\rho_2\oplus p_*\colon L'_2(\Zpi) \to L'_2(\Zhat\pi) \oplus L'_2(\bZ\pi^{\ab})$$
 induced by $\rho_2$ and the projection $p \colon \pi \to \pi^{\ab}$. Moveover, the composite
$\rho_2\circ \kappa'_4 = 0$ for abelian groups since $H^1(\Wh'(\Zhat\pi)) = 0$, and  $\rho_2$ is injective for $\pi$ abelian (see \cite[5.2.2]{Wall:1976a}). 

Therefore the computation of $\kappa'_4(\pi)$ is reduced to determining $\lambda_4(\pi)$, and 
the fact that $\kappa'_4(\pi)$ vanishes on integral homology classes follows from Corollary \ref{cor:fivethree}. 
Finally,  there is an injection $L'_2(\Zpi)  \to L^h_2(\Zpi) $ and hence $\kappa^h_4 \neq 0$ if and only if $\kappa'_4\neq 0$, so it is enough to determine $\kappa'_4(\pi)$.
This reduces the original oozing conjecture that $\kappa^h_j = 0$, for $j \geq 4$, to an explicit group homology calculation, namely computing the map $\lambda_4(\pi)$.  
\end{proof}

\subsection{Necessary conditions for $\beta\circ s_* \neq 0$}
\label{subsec:fiveone}

\begin{definition} For any internal direct sum 
$\pi^{\ab} = C_1\oplus C_2\oplus \dots \oplus C_n$ of non-trivial cyclic subgroups, the \emph{associated basis} 
 $\cB=\{v_1,v_2,\dots v_n\}$ for
$H^0(\pi^{\ab})$ is the linearly independent set consisting of the elements $v_i \in C_i$ of order two, for $1\leq i \leq n$.
\end{definition}

In any such decomposition of  $\pi^{\ab}$  the orders of the $C_i$, and the number of cyclic subgroups of the same order, are invariants of $\pi$. However, the factors $C_i \subseteq \pi^{\ab}$ are not unique as subgroups of $\pi$.

\begin{lemma}\label{lem:cyclic_basis} Suppose that  $H^1(Wh'(\Zhat\pi))\neq 0$.Then there exist
cyclic subgroups $\{C_i\}$ of $\pi^{\ab}$, such that
$\pi^{\ab} =C_1\oplus C_2\oplus \dots \oplus C_n$,  where the associated basis $\cB=\{v_1,v_2,\dots v_n\}$ for 
$H^0(\pi^{\ab})$ satisfies:
\begin{enumerate}
\item For some $ k\leq n$, the set $\{\delta(v_1),\dots, \delta(v_k)\}$ is a basis for $H^1(Wh'(\Zhat\pi))$;
\item The set $\{v_{k+1},\dots, v_n\} \subseteq \cB$ is a basis for $K$.
\end{enumerate} 
\end{lemma}

\begin{proof}  Let $\{ C_i \vv 1 \leq i \leq n\}$ denote a collection of cyclic subgroups of $\pi^{\ab}$ giving an internal direct sum decomposition. For each factor, pick a generator so that $C_i = \la g_i \ra$.  If $|C_i| = m_i$, for $1 \leq i \leq n$, we can choose the ordering so that
$m_1 \geq m_2 \geq  \dots \geq m_n >1$.

 Let $\{w_1, w_2, \dots, w_m\}$, with $0 <m \leq n$,  be a basis for $W:= H^1(Wh'(\Zhat\pi)\cong (\cy 2)^m$  over $\bF_2$.
 In terms of the basis $\{v_1, v_2, \dots, v_n\}$  for $V: = H^0(\pi^\ab)$, we can express the linear map $\delta\colon V \to W$ as  an $m \times n$ matrix $M = (a_{ij})$ of rank $m$ such that 
$$w_i = \sum_{i = 1}^m a_{ij} v_j\ .$$
By using row operations, we can assume that $M$ is in reduced row echelon form. Moreover, we are also allowed to use elementary column operations of the form $ \text{col}_j \to \text{col}_j  +  \text{col}_i$, if $j > i$. For example, if $ \text{row}_i$  has its first non-zero entry in column $k$,  and $a_{ij} = 1$ for some $j > k$, we can replace the factor $C_j = \la g_j\ra$  by $C'_j = \la g_j g_k^t\ra$, where $t= m_k/m_j$. After repeating these column operations, we may assume that $M$ is  only a single non-zero entry in each row. By re-ordering the cyclic factors, we obtain an associated basis $\cB$ of the required form.
\end{proof}

 Suppose that $\pi$ is a $2$-group such that  $\beta \circ
s_*(\wx )\neq 0$ for some $\wx  \in H_4(\pi;\cy 2)$. Let 
$\pi^{\ab} =C_1\oplus C_2\oplus \dots \oplus C_n$ be a decomposition of $\pi^{\ab}$ with the properties given in Lemma \ref{lem:cyclic_basis}. By re-ordering the cyclic factors if necessary, we may suppose that 
$p_*(\beta \circ s_*(\wx ))\neq 0 \in H^0(C_1; \cy 2)$, where $p \colon \pi \to C_1$ denotes projection onto the first factor. In particular, the preimage
 $N \leq \pi$ of $C_2\oplus
\dots \oplus C_n$ under the natural projection from $\pi$  to $\pi^{\ab}$  contains the preimage of $K$.

Let $g_1$ be an element in $\pi$ such that $[g_1]$ generates $C_1$ in
$\pi^{\ab}$. Let $\widehat C_1$ denote the subgroup of $\pi $ generated by $g_1$.
Since $v_1$ is in $C_1$, there exists a positive integer $r$  such that
$[g_1]^r=v_1$. Hence $g_1^r$ is not conjugate to its inverse in $\pi$.
Therefore the order of $g_1$ is strictly greater than $2r$. This means
$|\widehat C_1|>|C_1|$. Also notice that $\pi/N$ is a cyclic group and $|\pi/N|=|C_1|$.
Now let $p\colon \pi \to \pi/N$ denote the natural projection  and
consider the diagram

\eqncount
\begin{equation}\label{eq:fivesix}
\vcenter{\xymatrix{ H_4(\pi;\cy 2) \ar[d]_{p_*} \ar[r]^{s_*} &H_2(\pi;\cy
2)
\ar[d]_{p_*} \ar[r]^{\beta} &
             H^0(\pi^{\ab}) \ar[d]_{p_*}  \\
             H_4(\pi/N ;\cy 2)  \ar[r]^{s_*} &H_2(\pi/N ;\cy 2)
\ar[r]^{\beta} &
             H^0(\pi/N)
     }}
  \end{equation}
  Since $\pi/N$ is cyclic, the maps $s_*$ and $\beta$ in the lower row are both isomorphisms.
     Then we have  $p_* \circ \beta \circ s_*(\wx )\neq 0$, and hence $p_*\colon H_4(\pi;\cy 2) \to  H_4(\pi/N;\cy 2)$ is 
     non-zero (and conversely). Therefore
$p_*\colon H_2(\pi;\cy 2)\rightarrow H_2(\pi/N ;\cy 2)$ is also nonzero.
Dually this means $p^*\colon H^2(\pi/N ;\cy 2) \rightarrow H^2(\pi;\cy
2)$ is nonzero. 

\smallskip
We can summarize this last point as follows:
\begin{lemma}\label{lem:fivetwo}
  If $(\beta \circ s_*) (\wx ) \neq 0$, for some $\wx  \in H_4(\pi;\cy 2)$, then there exists a cyclic quotient $p\colon \pi \to C_1$ such that $0 \neq p_*(\wx ) \in H_4(C_1;\cy 2)$.  Moreover, 
  \begin{enumerate}
  \item $0 \neq \wx \cap \theta \in H_2(\pi; \cy 2)$, where
 $\theta = p^*(\bar\theta) \in H^2(\pi;\cy 2)$ is the pullback of the generator of $H^2(C_1; \cy 2)$.
 \item  If $\pi/N = C_1 \cong \cy 2$, we have $\theta = \phi^2$ where  $\phi = p^*(\phi_1) \in H^1(\pi; \cy 2)$ is the pullback of the generator $\phi_1 \in H^1(C_1; \cy 2)$.
 \end{enumerate}
 \end{lemma}
 \begin{corollary}\label{cor:fivethree} If $ \wx \in \Image\{H_4(\pi; \bZ) \to H_4(\pi; \cy 2)\}$, then  $(\beta \circ s_*) (\wx ) =0$. 
 \end{corollary}
 \begin{proof} If $ \wx \in \Image\{i_*\colon H_4(\pi; \bZ) \to H_4(\pi; \cy 2)\}$,  then the composite $p_* \circ i_*$ factors through $H_4(\pi/N; \bZ) = 0$,  for any cyclic quotient $\pi/N$. Hence 
 $p_*(\wx ) = 0$ and  $(\beta \circ s_*) (\wx ) = 0$ by Lemma \ref{lem:fivetwo}. 
  \end{proof}
\subsection{A sufficient condition for $\delta \circ \beta\circ s_*
\neq 0$}\label{subsec:fivetwo}   
 We assume that $H^1(Wh'(\Zhat\pi)) \neq 0$ and that
 $$\pi^{\ab} = C_1 \oplus C_2 \oplus \dots \oplus C_n$$ is an internal direct sum satisfying the conditions of Lemma \ref{lem:cyclic_basis}. We will call this an \emph{adapted decomposition of $\pi^{\ab}$}. 
 Let $p \colon \pi \to C_1$ be the first factor projection with respect to an adapted decomposition of $\pi^{\ab}$, and let $N = \ker p$.  Since $[\pi, \pi] \leq N$. we have  a natural map $N/[\pi,\pi] \to \pi^{\ab}$.
If $H^1(Wh'(\Zhat\pi)) \neq 0$, then  $\delta(v_1) \neq 0$, and $K = \ker \delta$ is contained in the image of  the induced map $H^0(N/[\pi, \pi]) \to H^0(\pi^\ab)$.
We
have the following lemma that gives a sufficient condition for  $\delta
\circ \beta\circ s_* \neq 0$.

\begin{lemma}\label{lem:suff_cond}   Let $p \colon \pi \to C_1$ be the first factor projection with respect to an adapted decomposition of $\pi^{\ab}$,  such that $\delta(v_1) \neq 0$, and let $N = \ker p$.
      If $p^* (x) \neq 0$, for some $x \in H^4(\pi/N;\cy 2)$, then 
      $(\delta\circ \beta\circ s_*) (\wx )\neq 0$, for some
$\wx  \in H_4(\pi;\cy 2)$. 
\end{lemma}
\begin{proof}
 First, notice that $H^4(\pi/N;\cy 2)\cong \cy 2$, hence $p_*\colon
H_4(\pi;\cy 2)\rightarrow  H_4(\pi/N;\cy 2)$ is nonzero if and only if
$p^*\colon H^4(\pi/N;\cy 2)\rightarrow  H^4(\pi;\cy 2)$ is nonzero.
Hence the existence of $x \in H^4(\pi/N;\cy 2)$ implies that there exists $\wx  \in 
H_4(\pi;\cy 2)$ such that $0 \neq p_*(\wx ) \in H_4(\pi/N; \cy 2)$. 

Second, notice that for all $\wz\in H^0(\pi^{\ab})$ we have $\delta(\wz)\neq
0$ whenever $p_*(\wz)\neq 0$. To see this, let $L :=\ker p_* \subset H^0(\pi^{\ab})$, $\bar L := L/K\cap L$,
and consider the diagram 
\eqncount
\begin{equation}\label{eq:detect}
\vcenter{\xymatrix@C-8pt@R-8pt{&0 \ar[d] &0 \ar[d] & 0 \ar[d]&\\
0 \ar[r]&K \cap L \ar[r] \ar[d]& L \ar[d] \ar[r] & \bar L\ar[d] \ar[r] & 0\\
0 \ar[r]& K \ar[r] \ar[d]& H^0(\pi^{\ab}) \ar[r]^(0.4){\delta}\ar[d]^{p_*}&  H^1(Wh(\Zhat\pi)) \ar[r]\ar[d] & 0\\
0 \ar[r]&K/K\cap L\ar[r]\ar[d]& H^0(\pi/N) \ar[r] \ar[d]& V \ar[d] \ar[r] & 0\\
&0& 0 & 0 &\\
}}
\end{equation}
 Note that  $H^0(N/[\pi, \pi]) = L$, so that  $K \subseteq L$ by construction. Hence $K\cap L = K$, $K/K\cap L = 0$, and $H^0(\pi/N)\cong V$, so that $L \cong \bar L$. Now we claim that $p_*(\wz) \neq 0$, for $\wz = \beta\circ s_*(\wx)$.
 Since $\pi/N =C_1$ is cyclic, by assumption, this follows by the diagram \eqref{eq:fivesix}
 together with the fact that the composite 
$\beta\circ s_*\colon H_4(\pi/N; \cy 2) \to H^0(\pi/N)$ is an isomorphism. Therefore $(\delta\circ \beta\circ s_*) (\wx )\neq 0$.
\end{proof}

\subsection{An example with $\delta \circ \beta \circ s_*\neq 0$  }\label{subsec:fivethree}

Let $ V \overset{ i }{ \rightarrowtail  } \pi  \overset{ \alpha }{
      \twoheadrightarrow } W   $ be a group extension  and $M$ a
$\pi$-module.
Then we have the Lyndon-Hochschild-Serre spectral sequence $
\{E^{n,m}_{r}(\alpha,M),d_{r}\}$ which converges to $H^{*}(\pi;M)$ and
has second page given by:
\begin{equation*}
      E^{n,m}_{2}(\alpha,M)=H^{n}(W ;H^{m}(V ;M)).
\end{equation*}

Let $\pi:= G(16384)$ be the group of order $2^{14} = 16384$
given by the following presentation
\begin{equation*}
      \pi= \left\langle
      x_1,x_2,x_3,x_4,x_5,x_6,x_7,z_1,z_2,z_3,z_4,z_5,z_6,z_7 \vv r\in R
\right\rangle
\end{equation*}
where $R$ contains the relations
$x_i^2 = z_i$,
$z_i^2 = 1$,
$[x_i,z_j]=1$,
$[z_i,z_j]=1$,
$[x_i,x_1]=1$ for all $i, j\in \{1,2,3,4,5,6,7\}$,
$[x_i,x_j]=1$ for all $i \in \{2,3\}$, $j\in \{4,5,6,7\}$,
$[x_i,x_j]=1$ for all $i \in \{4,5\}$, $j\in \{6,7\}$,
$[x_2,x_3]=z_1$, $[x_4,x_5]=z_2$, $[x_6,x_7]=z_3$.
Notice that we have a central group extension  $ V \overset{ i }{
      \rightarrowtail  } \pi  \overset{ \alpha }{
      \twoheadrightarrow } W   $ where both
\begin{equation*}
      V=\left\langle z_1,z_2,z_3,z_4,z_5,z_6,z_7 \right\rangle
      \cong (\cy 2 )^{7}
\end{equation*}
and
\begin{equation*}
      W=\pi/V=\left\langle [x_1],[x_2],[x_3],[x_4],[x_5],[x_6],[x_7]
      \right\rangle
      \cong (\cy 2 )^{7}
\end{equation*}
are elementary abelian  (see  \cite{Pakianathan:2012} for information about the spectral sequence calculations in this setting).

\begin{corollary}\label{cor:ooze_example} The group $\pi = G(16384)$ has $\kappa'_4(\pi) \neq 0$.  
\end{corollary}      
\begin{proof}
We start by recording a GAP computation, showing that
$$H^1(\Wh'(\hat{\bZ}_2\pi))\cong \frac{\{ g\in \pi_{\textup{\ab}}  \vv
      g^2=1 \}}{\langle g \vv g\sim g^{-1}\rangle }\cong \frac{\langle
      [x_1],[x_2],[x_3],[z_4],[z_5],[z_6],[z_7] \rangle }{
      \langle [z_4],[z_5],[z_6],[z_7] \rangle}\cong (\cy 2 )^{3}$$

The Lyndon-Hochschild-Serre spectral sequence for the
group extension
$ V \overset{ i }{ \rightarrowtail  } \pi  \overset{ \alpha }{
      \twoheadrightarrow } W $
implies that
\begin{equation*}
      E^{0,1}_{2}(\alpha,\cy 2)=H^{1}(V ;\cy 2)=V^*=\Hom(V,\cy 2)=\langle
      \zeta_1, \zeta_2, \zeta_3, \zeta_4, \zeta_5, \zeta_6, \zeta_7\rangle
\end{equation*}
where $\zeta_i=z_i^*$ for $i\in \{1,2,3,4,5,6,7\}$.  Considering $$H^1(W
;\cy
2)=W^*=\langle X_1,X_2,X_3,X_4,X_5,X_6,X_7\rangle $$ where $X_i=[x_i]^*$
for $i\in
\{1,2,3,4,5,6,7\}$. We have
\begin{equation*}
      \begin{aligned}
              d_2(\zeta_1)  & = X_1^2 + X_2X_3 \\
              d_2(\zeta_2)  & = X_2^2 + X_4X_5 \\
              d_2(\zeta_3)  & = X_3^2 + X_6X_7 \\
          d_2(\zeta_4)  & = X_4^2 \\
              d_2(\zeta_5)  & = X_5^2 \\
          d_2(\zeta_6)  & = X_6^2 \\
          d_2(\zeta_7)  & = X_7^2 \\
      \end{aligned}
\end{equation*}
Let $I_2$ be the ideal generated by the image of $d_2$. Then we have
\begin{equation*}
      \begin{aligned}
              d_3(\zeta^2_1)  & = X_2^2X_3 + X_2X_3^2 +I_2\\
              d_3(\zeta^2_i)  & = 0 +I_2\\
      \end{aligned}
\end{equation*}
for all $i$ in $\{2,3,4,5,6,7\}$.
Let $I$ be the kernel of the edge homomorphism from  $H^*(W ;\cy
2)=E^{*,0}_{2}(\alpha,\cy 2)$ to  $H^*(\pi ;\cy 2)$.  By Lemma \ref{lem:fivefive} below,  it follows that  $X_1^4$ is not in
$I$. We have $$X_2^4=(d_2(\zeta_2))^2 + X_5^2d_2(\zeta_4) $$
$$X_3^4=(d_2(\zeta_3))^2 + X_6^2d_2(\zeta_7) $$
Hence $X_i^4$ is in $I$ for all $i$ in $\{2,3,4,5,6,7\}$.

\medskip
As in Section \ref{subsec:fivetwo}, 
let $N=\langle z_1, x_2, x_3, x_4, x_5, x_6, x_7\rangle$, and let $p \colon \pi \to \pi/N$ denote the quotient map. 
Since $d_2$ is injective, $H^1(W; \cy 2) \xrightarrow{\ \alpha^*\ } H^1(\pi; \cy 2)$ is an isomorphism, and  there is a unique class $\bar X^1 \in H^1(\pi; \cy 2)$ such that $p^*(\bar X_1) = \alpha^* (X_1) \in H^1(\pi; \cy 2)$.  We let $x:= \bar X_1^4 \in H^4(\pi/N; \cy 2)$ for use in  Lemma \ref{lem:suff_cond}. The following result checks that $X_1^4$ is not in the ideal $I$, and hence $0 \neq p^*(x) \in H^4(\pi; \cy 2)$, implying that $\delta \circ
\beta\circ s_* (\wx )\neq 0$ for some $\wx \in H_4(\pi; \cy 2)$.  \end{proof}

\begin{lemma}\label{lem:fivefive}
Let 
$$\alpha =X_1^2 + X_2X_3,$$ 
$$\beta =X_2^2 + X_4X_5,$$ 
$$\gamma = X_3^2 + X_6X_7,$$
$$\theta = X_2^2X_3+X_2X_3^2 $$
be  in $R=\mathbb{F}_2[X_1,X_2,\dots,X_7]$. Then $X_1^4$ is not in the ideal generated by $\alpha $, $\beta $, $\gamma $, $\theta$, $X_4^2$, $X_5^2$, $X_6^2$, $X_7^2$, $X_4^2X_5+X_4X_5^2$, $X_6^2X_7+X_6X_7^2$.
\end{lemma}
\begin{proof}
Suppose otherwise. Then
\eqncount
\begin{equation} \label{eq:main_in_lem}
X_1^4=f\alpha +g\beta +h\gamma + k\theta + z
\end{equation}
for some $f$, $g$, $h$, $k$ in $R$ and for some $z$ in the ideal generated by the elements $X_4^2$, $X_5^2$, $X_6^2$, $X_7^2$, $X_4^2X_5+X_4X_5^2$, $X_6^2X_7+X_6X_7^2$. Assume that $m_1$, $m_2$, ..., $m_k$ are distinct monomials in $X_1$, $X_2$,  ..., $X_7$. In the rest of the proof, when we write 
$$w=c_1m_1+c_2m_2+\cdots +c_km_k + \bigstar $$ 
it will mean that 
$$w=c_1m_1+c_2m_2+\cdots +c_km_k +  \widetilde{w}$$ for some $\widetilde{w}$ which has no $m_1$, $m_2$, ... $m_k$ terms. For the rest of the proof, the coefficients $\{ c_i\}$ are elements in $\mathbb{F}_2$. Assume that 
$$f=c_1X_1^2+c_2X_2X_3+ c_3X_2^2 +c_4X_3^2+c_5X_4X_5+c_6X_6X_7+ \bigstar ,$$
and
$$g=c_7X_1^2+c_8X_3^2+c_9X_2X_3+c_{10}X_6X_7+\bigstar$$
and
$$h=c_{11}X_1^2+c_{12}X_2^2+c_{13}X_2X_3+c_{14}X_4X_5+\bigstar$$
and
$$k=c_{15}X_2+c_{16}X_3+\bigstar$$
Let  $m$ be a monomial in  $X_1$, $X_2$,  ..., $X_7$. We will write $(m)$ to denote the equation on $c_i$'s that we get by considering the coefficient of the monomial $m$ in the  Equation \ref{eq:main_in_lem}.
$$
\begin{array}{rcll}
1 & = & c_1 & (\,  X_1^4  \,) \\[0.8ex]
0 & = & c_2+c_8+c_{12}+c_{15}+c_{16} & (\,  X_2^2X_3^2  \,) \\[0.8ex]
0 & = & c_3+c_9+c_{15} & (\,  X_2^3X_3  \,) \\[0.8ex]
0 & = & c_4+c_{13}+c_{16} & (\,  X_2X_3^3  \,) \\[0.8ex]
0 & = & c_5+c_7 & (\,  X_1^2X_4X_5  \,) \\[0.8ex]
0 & = & c_5+c_9 & (\,  X_2X_3X_4X_5  \,) \\[0.8ex]
0 & = & c_1+c_2 & (\,  X_1^2X_2X_3  \,) \\[0.8ex]
0 & = & c_8+c_{14} & (\,  X_3^2X_4X_5  \,) \\[0.8ex]
0 & = & c_3+c_7 & (\,  X_1^2X_2^2  \,) \\[0.8ex]
0 & = & c_4+c_{11} & (\,  X_1^2X_3^2  \,) \\[0.8ex]
0 & = & c_6+c_{11} & (\,  X_1^2X_6X_7  \,) \\[0.8ex]
0 & = & c_6+c_{13} & (\,  X_2X_3X_6X_7  \,) \\[0.8ex]
0 & = & c_{10}+c_{12} & (\,  X_2^2X_6X_7  \,) \\[0.8ex]
0 & = & c_{10}+c_{14} & (\,  X_4X_5X_6X_7  \,) 
\end{array} 
$$
By the last 10 equations above we have  $c_1=c_2$ and $c_3=c_5=c_7=c_9$ and $c_4=c_6=c_{11}=c_{13}$ and $c_8=c_{10}=c_{12}=c_{14}$. By the first equation we get $c_1=c_2=1$. Now by the second equation we get  $0=1+c_{15}+c_{16}$ because $c_8=c_{12}$. By third equation we get $c_{15}=0$ because $c_3=c_9$. By fourth equation we get $c_{16}=0$ because $c_4=c_{13}$.
Hence $0=1+c_{15}+c_{16}=1+0+0=1$. This is a contradiction.   
\end{proof}

\begin{remark} By GAP computations we have verified that $\kappa'_4(\pi) = 0$, for all elementary by elementary $2$-groups of order $\leq 256$, and for about a third of such groups of order $512$.
\end{remark}

\subsection{An example with $\beta \circ s_*\neq 0$ and $\delta \circ
\beta\neq 0$ } 
Here we give an example to show that both $\beta \circ s_*$ and $\delta
\circ \beta$ could be nonzero, but the composite $\delta
\circ \beta\circ s_* = 0$. Such examples illustrate the difficulty of finding groups with $\kappa'_4(\pi) \neq 0$.
The spectral sequence calculations follow the same methods as above

\medskip
Let $\pi$ be the group of order $256$, identified as  $SG[256,9039]$ in the Small Groups Library. It is
given by the following presentation
\begin{equation*}
      \pi= \left\langle
      x_1,x_2,x_3,x_4,x_5,x_6,x_7,x_8 \vv r\in R \right\rangle
\end{equation*}
where $R$ contains the relations
$x_1^2 = x_5 x_7 x_8$,
$x_2^2 = x_6$,
$x_3^2 = x_5x_6$,
$x_4^2 = x_5x_6$,
$x_5^2=x_6^2=x_7^2=x_8^2=1$,
$[x_1,x_2] = x_5$,
$[x_1,x_3] = x_6$,
$[x_1,x_4]= x_8$,
$[x_2,x_3]= x_7$,
$[x_2,x_4]= x_5$,
$[x_3,x_4] = 1$, and
$[x_i,x_5]=[x_i,x_6]=[x_i,x_7]=[x_i,x_8]$ for every $i\in
\{1,2,3,4,5,6,7\}$.
Notice that we have a central group extension  $ V \overset{ i }{
\rightarrowtail  } \pi  \overset{ \alpha }{
      \twoheadrightarrow } W   $ where
\begin{equation*}
      V= Z(\pi) = [\pi,\pi]=\left\langle x_5,x_6,x_7,x_8 \right\rangle
      \cong \cy 2 \times \cy 2 \times \cy 2 \times \cy 2
\end{equation*}
and
\begin{equation*}
      W= \pi^{\ab}=\pi/[\pi,\pi]=\left\langle [x_1],[x_2],[x_3],[x_4]
\right\rangle
      \cong \cy 2 \times \cy 2 \times \cy 2 \times \cy 2.
\end{equation*}
Also note that the GAP code in Appendix \ref{sec:appone} shows that
$$H^1(\Wh'(\hat{\bZ}_2\pi))\cong \frac{\{ g\in \pi^{\ab}  \vv
g^2=1 \}}{\langle g \vv g\sim g^{-1}\rangle }\cong \frac{\langle
[x_1],[x_2],[x_3],[x_4] \rangle }{
\langle[x_1]+[x_2],[x_2]+[x_3],[x_3]+[x_4]\rangle}\cong \cy 2. $$
By considering the Lyndon-Hochschild-Serre spectral sequence for the
group extension
$ V \overset{ i }{ \rightarrowtail  } \pi  \overset{ \alpha }{
\twoheadrightarrow } W $
we get
\begin{equation*}
      E^{0,1}_{2}(\alpha,\cy 2)=H^{1}(V ;\cy 2)=V^*=\Hom(V,\cy 2)=\langle
\zeta_5,\zeta_6,\zeta_7,\zeta_8\rangle
\end{equation*}
where $\zeta_i=x_i^*$ for $i\in \{5,6,7,8\}$. Considering $$H^1(W ;\cy
2)=W^*=\langle X_1,X_2,X_3,X_4\rangle $$ where $X_i=[x_i]^*$ for $i\in
\{1,2,3,4\}$. We use a GAP program to obtain the following results: we have
\begin{equation*}
      \begin{aligned}
              d_2(\zeta_5)  & = X_1^2 + X_1X_2 +  X_2X_4 + X_3^2 + X_4^2\\
              d_2(\zeta_6)  & = X_1X_3 + X_2^2 + X_3^2 + X_4^2 \\
              d_2(\zeta_7)  & = X_1^2 + X_2X_3 \\
              d_2(\zeta_8)  & = X_1^2 +  + X_1X_4  .
      \end{aligned}
\end{equation*}
Let $I_2$ be the image of $d_2$.  Then we have
\begin{equation*}
      \begin{aligned}
              d_3(\zeta^2_5)  & = X_1^2X_2 +X_1X_2^2 + X_2^2X_4 + X_2X_4^2 +I_2\\
              d_3(\zeta^2_6)  & =  X_1^2X_3 + X_1X_3^2 +I_2\\
              d_3(\zeta^2_7)  & =  X_2^2X_3 + X_2X_3^2 +I_2\\
              d_3(\zeta^2_8)  & = X_1^2X_4 + X_1X_4^2 +I_2
      \end{aligned}
\end{equation*}
Let $I$ be the kernel of the edge homomorphism from  $H^*(W ;\cy 2)=E^{*,0}_{2}(\alpha,\cy 2)$ to  $H^*(\pi ;\cy 2)$.  Then it is
straight forward to check $X^4_2+X^4_3$ and $X^4_3+X^4_4$ are not in
$I$. It is also  straight forward to check that $X^4_1$ and
$X^4_2+X^4_3+X^4_4$  are in $I$. Hence the image of $s_*$  in $H_4(\pi; \cy 2)$ is generated
by the duals of the cohomology classes  $X^2_2+X^2_3+I$  and $X^2_3+X^2_4+I$. 

 It follows that the image of  $\beta \circ
s_*$ is generated by $[x_2]+[x_3]$ and $[x_3]+[x_4]$, and hence nonzero
in   $H^0(\pi^{ab})$. Moreover $\delta \circ \beta([x_i])\neq 0$ for all
$i$ in $\{1,2,3,4\}$, implying that $\delta\circ \beta\circ s_* = 0$.

\section{More information about  $\kappa_4(\pi)$}\label{sec:six}
In Theorem E we have provided an explicit algebraic condition to check for non-vanishing examples for which $\kappa_4'(\pi) \neq 0$, and Corollary \ref{cor:ooze} provides one such example,  in which the group $\pi$ is elementary abelian by elementary abelian. We first present a conjecture for the vanishing of $\kappa_4'(\pi ) $ in case $\pi$ does not satisfy this  assumption. In the last sections we show how to obtain further non-vanishing $\kappa^s_4$ examples (see Proposition \ref{prop:compat_pair}, Proposition \ref{prop:kappas_ex2}, and  and Remark \ref{rem:moreex}).

\subsection{A conjecture about  $\kappa_4'(\pi) $.}
Note that if $\kappa_4'(\pi) \neq 0$, then by Theorem E the composite
$\delta\circ\beta\circ s_* \neq 0$ for the group $\pi$. 
In particular, if $\delta\circ\beta\circ s_*(\wx )\neq 0$, for some $\wx  \in H_4(\pi;\cy 2)$,  then by Lemma \ref{lem:fivetwo} there must exist a cyclic quotient  $p\colon \pi \to \pi/N = C_1$ such that $0 \neq p_*(\wx ) \in H_4(C_1;\cy 2)$. The following conjecture is motivated by trying to show that  $C_1$ must have order two if
$\kappa_4'(\wx ) \neq 0$.

\begin{question} If $\kappa'_4(\pi) \neq 0$, does it follow that $\pi^{\ab}$ is elementary abelian ?
\end{question}

If $\pi^\ab$ has cyclic factors of orders $\geq 4$ in an internal direct sum decomposition, then we can use an filtration approach to study the map 
$\delta\circ\beta\circ s_*$.

\begin{conjecture}\label{conj:subquotients} Let $N \triangleleft T \trianglelefteq W \triangleleft \pi$ be an increasing sequence of normal subgroups of a $2$-group $\pi$ such that  $\pi/N$ is a cyclic group, $T/N\cong \cy 2$ and $\pi/W\cong \cy 2$. Let $i\colon W\rightarrow \pi $ be the inclusion and $p\colon \pi\rightarrow \pi/N$ be the natural quotient map. 
Assume that  the maps
\begin{itemize} \setlength{\itemsep}{0pt plus 8pt}  
\item $p_*\colon H_4(\pi;\cy 2) \rightarrow  H_4(\pi/N;\cy 2)\cong \cy 2$, and 
\item $(p\circ i)_*\colon H_2(W;\cy 2) \rightarrow  H_2(\pi/N;\cy 2)\cong \cy 2$ 
\end{itemize}
 are both  non-zero, but $(p\circ i)_*\colon H_4(W;\cy 2) \rightarrow  H_4(\pi/N;\cy 2)\cong \cy 2$ is zero. Then the number of $\pi$-conjugacy classes of elements in $T-N$ is odd.
\end{conjecture}

 \begin{remark} In Appendix \ref{app:appthree} we provide some GAP code to test this conjecture. The smallest groups that contains such an increasing sequence are $SG[32,8]$, $SG[64,13]$ and $SG[64,14]$. In these groups, each such increasing sequence of subgroups $T-N$ has an odd number of $\pi$-conjugacy classes, verifying the conjecture.  
 \end{remark}

Recall the formula:
$$H^1(\Wh'(\hat{\bZ}_2\pi))\cong \frac{\{ [g]\in \pi^{\ab}  :
[g]^2=1 \}}{\langle [g] : g\sim g^{-1}\rangle }\cong S/C,$$
where $S=\{ g\in \pi  \vv
g^2\in[\pi,\pi] \}$  and $C=\langle g\in S \vv g\sim g^{-1}\text{ or } g\in
[\pi,\pi]\rangle$. Now let $$T=\langle g\in S  \vv g\in
C, \text{ or } 
t^2=g \text{ for some \ } t\in \pi\rangle,$$  and notice that $T/C$ is a subgroup of $H^1(\Wh'(\hat{\bZ}_2\pi))$. For example, when $\pi=SG(32,4)$ we have $S=T$. 
 
\begin{theorem} Let $\pi$ be a finite $2$-group.
Assume that Conjecture \textup{\ref{conj:subquotients}} is true and $\kappa'_4(H)=0$,  for every proper subgroup $H$ of $\pi$. Then the intersection of the image of $\kappa'_4(\pi)$ with $T/C$ is trivial. 
\end{theorem}
   \begin{proof}
      Suppose otherwise.  Then there exists an element $\wx \in H_4(\pi; \cy 2)$, and  an adapted decomposition of $\pi^\ab$ such that  $0 \neq p_*(\beta\circ s_*(\wx ))$ under the first factor projection $p\colon  \pi \to C_1$,  for which $\delta (v_1)\neq 0$ is in $T/C$. Let $N$ be as in Section \ref{subsec:fivetwo}. Since $v_1$ is in $T$, $\pi/N$ is cyclic group of order at least $4$. In fact there exists  an increasing sequence  $N \triangleleft T \trianglelefteq W \triangleleft \pi$  of normal subgroups of the $2$-group $\pi$. Notice that we have $p_*\colon H_4(\pi;\cy 2) \rightarrow  H_4(\pi/N;\cy 2)\cong \cy 2$ non-zero by Lemma \ref{lem:fivetwo} because  $\delta (v_1)\neq 0$. 
      
      Moreover, $(p\circ i)_*\colon H_2(W;\cy 2) \rightarrow  H_2(\pi/N;\cy 2)\cong \cy 2$  is non-zero, since the composite $$H_2(T;\cy 2) \to H_2(\pi;\cy 2) \to H^0(\pi^{\ab}) \to H^0(\pi/N)$$
       is surjective, and $H^0(T/N) \cong H^0(\pi/N)$. This uses the fact that $H_2(T;\cy 2) \to H_2(\pi;\cy 2)$ factors through the map $ H_2(W;\cy 2)  \to H_2(\pi;\cy 2)$, and the isomorphisms
$$H_2(T/N;\cy 2) \cong H_2(W/N;\cy 2) \cong H_2(\pi/N;\cy 2)$$ induced by the inclusions.
      
      But $(p\circ i)_*\colon H_4(W;\cy 2) \rightarrow  H_4(\pi/N;\cy 2)\cong \cy 2$ is zero because $\kappa'_4(W)=0$.  Then the number of $\pi$-conjugacy classes of elements in $T-N$ is odd. Consider the action on the set $\pi$-conjugacy classes of elements in $T-N$ given by taking inverses. This action must have a fixed point and hence contradicting $\delta (v_1)\neq 0$.   
         \end{proof}

\begin{corollary} Assume that Conjecture \textup{\ref{conj:subquotients}} is true, and that $\pi^{\ab}$ has no direct factors of order two. Then $\kappa'_4(\pi) = 0$.
\end{corollary}

\subsection{Codimension two twisting diagrams}\label{subsec:codim_two_ex} Here is an approach to finding an example such that $\kappa^s_4 \neq 0$, but $\kappa'_4 = 0$. We look for a finite $2$-group $\pi$ equipped with certain  properties. For simplicity, we restrict to groups $\pi$  such that $\pi^{\ab}$ has exponent two. Recall from \cite[Theorem 3]{Oliver:1980a} that
there is a natural isomorphism
$$\omega\colon H_2(\pi;\bZ)/H_2^{\ab}(\pi) \xrightarrow{\cong} SK_1(\Zhat\pi)$$
where $H_2^{\ab}(\pi) \subseteq H_2(\pi; \bZ)$ denotes the image under induction from abelian subgroups of $\pi$, and a boundary map $\bd\colon H^1(SK_(\Zhat\pi)) \to L_{2i}(\Zhat\pi)$, for $i = 0,1 \pmod 4$.

\begin{definition}\label{def:kappalist}  Let $\pi$ be a non-abelian finite 2-group and $\theta \in H^2(\pi; \cy 2)$ a cohomology class. We say that the pair $(\pi,  \theta)$ is \emph{compatible} provided that the following conditions hold. 
\begin{enumerate}\setlength{\itemsep}{0pt plus 6pt}  
\item The abelianization $\pi^{\ab}$ has exponent two.
\item A central extension
$1 \to \sigma \to \tilde \pi \to \pi \to 1$,
  classified by  $\theta \in H^2(\pi;\cy 2)$, with $\sigma = \la t\ra\cong \cy 2$, defines a 2-cover $\alpha\colon \tilde \pi \to \pi$. We require that $t \in [\tilde\pi, \tilde\pi]$.
  \item 
$H^1(Wh'(\Zhat\tilde\pi)) = 0$.
\item  $SK_1(\Zhat\pi) \neq 0$, but the map 
  $SK_1(\Zhat\tilde\pi)\to SK_1(\Zhat\pi)$ is zero.
 \item The boundary map $\bd\colon H^1(SK_1(\Zhat\pi)) \to L_{0}(\Zhat\pi)$ is injective.
 \item There exists a class 
 $\wx \in H_4(\pi; \cy 2)$ such that    $0 \neq \wz = \wx \cap \theta \in H_2(\pi;\cy 2)$ is the image $\wz = i_*(\wz_1)$ of an integral class $\wz_1 \in H_2(\pi; \bZ)$.
 \item For the image $[\wz_1] \in H_2(\pi;\bZ)/H_2^{\ab}(\pi)$, we have
 $1 \neq \omega([\wz_1]) \in SK_1(\Zhat\pi)$.
\end{enumerate}
\end{definition}
\begin{remark}\label{rem:cohom_form}
Alternately,  by duality in terms of the cohomology Bockstein, for condition (vi) we need a class $ z \in H^2(\pi; \cy 2)$ such that $\beta^*( z) \neq 0$ and $z \cup \theta \neq 0$.
\end{remark}

\begin{proposition}
\label{prop:compat_pair} If $(\pi, \theta)$ is a compatible pair, then $\kappa^s_4(\pi) \neq 0$ and $\kappa'_4(\pi) = 0$.
\end{proposition}
\begin{proof} Note first that $H^1(Wh'(\Zhat\pi)) = 0$ by condition (iii). Hence $\kappa'_4(\pi) = 0$ by Theorem E and condition (i). Furthermore, $\kappa_2^s(\wz) \neq 0$ by conditions (iv), (v), and Theorem C applied to $\wz = \wx\cap \theta \in H_2(\pi; \cy 2)$.

 We have a central extension
 $1 \to \sigma \to \tilde \pi \to \pi \to 1$,
 such that $\sigma \subset [\tilde\pi, \tilde\pi]$, and by condition (ii) we have $H^0(\tilde\pi^{\ab}) \cong H^0(\pi^{\ab})$ under the projection map.
 The class $\theta \in H^2(\pi; \cy 2)$ defines a circle bundle
 $$ S^1 \to Y \to K(\pi, 1)$$
 inducing the given central extension on fundamental groups (see \cite[pp128-131]{Wall:1970}).  This circle bundle is orientable if and only if  $\theta$ is the image of an integral class,  or equivalently $\beta^*(\theta) = 0$, where $\beta^*\colon H^2(\pi; \cy 2) \to H^3(\pi; \bZ)$ is the Bockstein on cohomology.

 Hence we have a codimension 2 twisting diagram:

 \eqncount
\begin{equation}\label{eq:fiveeleven}
\vcenter{
\xymatrix@C+35pt{H_4(\pi;\cy 2) \ar[d]^{\cap\,\theta}\ar[r]^(0.5){\bar\kappa^s_4} &
L_2^s(\Zhat\pi) \ar[dr]^a&\\
H_2(\pi; \cy 2) \ar[r]^(0.5){\bar\kappa^s_2} & L_0^s(\Zhat\pi)\ar[r]^b&
L_2^s(\Zhat\tilde\pi\to \Zhat\pi)
}
}
\end{equation}
We claim that the map $b$ in this diagram is injective on the image of $\kappa^s_2$ restricted to $H_2(\pi;\bZ)\subseteq H_2(\pi;\cy 2)$ by the Ranicki-Rothenberg comparison sequences. We have an exact sequence
$$ \dots \to L^s_{n+3}(\tilde\pi \to \pi) \to LS_{n}(\Phi) \to L^s_n(\Zhat\pi) \to L^s_{n+2}(\tilde\pi \to \pi) \to \dots$$
arising from the circle bundle (see \cite[p127]{Wall:1970}). In our setting, 
$$LS_{n}(\Phi) = LN_n(\tilde\pi \to \pi) \cong L_n^s(\Zhat\tilde\pi, \beta, t),$$  and this long exact sequence can be identified with the sequence
$$ \dots \to L_3^s(\Zhat\tilde\pi\to \Zhat\pi) \xrightarrow{\ \bd\ } L_0^s(\Zhat\tilde\pi, \beta, t) \xrightarrow{\ \alpha_*\ } L_0^s(\Zhat\pi) \xrightarrow{\ b\ }  L_2^s(\Zhat\tilde\pi\to \Zhat\pi)$$
where $(\Zhat\tilde\pi, \beta, t)$ is the induced twisted antistructure and the middle map is induced by the projection $\alpha\colon \tilde\pi \to \pi$ (see Ranicki \cite[p805,808]{Ranicki:1981}). 
 Under our assumptions, $H^1(Wh'(\Zhat\tilde\pi)) = 0$  by condition (iii), and $\wpi^\ab \cong \pi^\ab$ implies that 
  $H^1(Wh'(\Zhat\pi)) = 0$. It follows that
 $L_0^s(\Zhat\pi) \cong H^1(SK_1(\Zhat\pi))$ and $L_0^s(\Zhat\tilde\pi, \beta, t)  \cong 
 H^1(SK_1(\Zhat\tilde\pi))$ (see  \cite[\S 8]{Hambleton:2000}). Since the map
 $SK_1(\Zhat\tilde\pi)\to SK_1(\Zhat\pi)$ is zero, we conclude that the $\alpha_* = 0$ and $b$ is injective.
 
Since $0 \neq \kappa_2^s(\wz) = \kappa_2^s(i_*(\wz_1)) \in L_0^s(\Zhat\pi)$, it follows that $\kappa_4^s(\wx) \neq 0$ by commutativity of the twisting diagram and the injectivity of the map $b$. 
\end{proof}

\subsection{An example which satisfies the hypothesis of Proposition \ref{prop:compat_pair}
}
We can use GAP to search for examples satisfying the conditions in Definition \ref{def:kappalist}. For finite $2$-groups of orders up to 128, we will use the cohomology calculations of Green and King \cite{Green:2011} (see the webpage \url{https://users.fmi.uni-jena.de/cohomology/}).  The spectral sequence calculations follow the same methods as in previous sections.
\begin{proposition}\label{prop:kappas_ex2} There exists a compatible pair $(\pi, \theta)$ for $\pi = SG[128,1376]$.
\end{proposition}
\begin{proof}
Let $\tilde \pi $ be the group of order $256$, identified as
$SG[256,8129]$ in the Small Groups Library. We will use the following
presentation of this group
\begin{equation*}
      \tilde  \pi= \left\langle
      \tilde x_1,\tilde x_2,\tilde x_3,\tilde x_4,\tilde x_5,\tilde
x_6,\tilde x_7,\tilde x_8 \vv r\in \tilde R \right\rangle
\end{equation*}
where $\tilde R$ contains the relations
$\tilde x_1^2 = \tilde x_6 \tilde x_7$,
$\tilde x_2^2 = \tilde x_6$,
$\tilde x_4^2 = \tilde x_5$,
$\tilde x_3^2 = \tilde x_5^2=\tilde x_6^2=\tilde x_7^2=\tilde x_8^2=1$,
$[\tilde x_1,\tilde x_2] = \tilde x_8$,
$[\tilde x_1,\tilde x_3] = \tilde x_5 \tilde x_6 \tilde x_8$,
$[\tilde x_1,\tilde x_4] = 1$,
$[\tilde x_2,\tilde x_3] = \tilde x_7$,
$[\tilde x_2,\tilde x_4] = 1$,
$[\tilde x_3,\tilde x_4] = \tilde x_5$, and
$[\tilde x_i,\tilde x_5] = [\tilde x_i,\tilde x_6]=[\tilde x_i,\tilde
x_7]=[\tilde x_i,\tilde x_8]=1$ for every $i\in
\{1,2,3,4,5,6,7\}$.

\smallskip
Let $\pi $ be the group of order $128$, identified as  $SG[128,1376]$ in
the Small Groups Library. This group has the following presentation
\begin{equation*}
       \pi= \left\langle
      x_1,x_2,x_3,x_4,x_5,x_6,x_7 \vv r\in R \right\rangle
\end{equation*}
where $R$ contains the relations
$x_1^2 = x_6 x_7$,
$x_2^2 = x_6$
$x_4^2 = x_5$,
$x_3^2=x_5^2=x_6^2=x_7^2=1$,
$[x_1,x_2] = x_5$,
$[x_1,x_3] = x_6$,
$[x_1,x_4]= 1$,
$[x_2,x_3]= x_7$,
$[x_2,x_4]= 1$,
$[x_3,x_4] = x_5$, and
$[x_i,x_5]=[x_i,x_6]=[x_i,x_7]=1$ for every $i\in
\{1,2,3,4,5,6,7\}$.
Now we define a surjective group homomorphism $\alpha$ from $\tilde \pi
$ to $\pi$ by sending $ \tilde x_1,\tilde x_2,\tilde x_3,\tilde
x_4,\tilde x_5,\tilde x_6,\tilde x_7,\tilde x_8$ to
$       x_1,x_2,x_3,x_4,x_5,x_6,x_7,x_5$ respectively. The kernel of this
group homomorphism is $$\sigma = \langle \tilde x_5\tilde
x_8\rangle\cong \cy 2.$$ Hence  we have a central extension
$$1 \to \sigma \to \tilde \pi \overset{\alpha}{\to} \pi \to 1.$$
Let
$$\tilde V=Z(\tilde \pi) = [\tilde \pi,\tilde \pi]=\left\langle \tilde
x_5,\tilde x_6,\tilde x_7,\tilde x_8 \right\rangle
\cong \cy 2 \times \cy 2 \times \cy 2 \times \cy 2$$
and
    $$V=Z(\pi) = [\pi,\pi]=\left\langle x_5,x_6,x_7 \right\rangle
\cong \cy 2 \times \cy 2 \times \cy 2 $$
and
$$\tilde W=\tilde \pi^{\ab}=\tilde \pi/[\tilde \pi,\tilde
\pi]=\left\langle [\tilde x_1],[\tilde x_2],[\tilde x_3],[\tilde x_4]
\right\rangle \cong \cy 2 \times \cy 2 \times \cy 2 \times \cy 2$$
and
$$W= \pi^{\ab}=\pi/[\pi,\pi]=\left\langle
[x_1],[x_2],[x_3],[x_4]
\right\rangle \cong \cy 2 \times \cy 2 \times \cy 2 \times \cy 2$$
where $\alpha $ induces an isomorphism from $\tilde \pi^{\ab}$
to $\pi^{\ab}$ which sends $[\tilde x_i]$ to $[x_i]$ for all $i$
in $\{1,2,3,4\}$. 
Now we have $gxg^{-1}=x^{-1}$ when 
$$(g,x) \in \left\{(1,\tilde x_3), (\tilde x_3 \tilde x_4, \tilde x_4), (\tilde x_1 \tilde x_2 \tilde x_4, \tilde x_2 \tilde x_3), (\tilde x_2 \tilde x_4, \tilde x_1 \tilde x_3), \right\}.$$
 Hence $[\tilde x_3]$, $[\tilde x_4]$,  $[\tilde x_2\tilde x_3]$, and $[\tilde x_1\tilde x_3]$ are all zero in   $H^1(\Wh'(\hat{\bZ}_2\tilde \pi))$, and  we have
$$H^1(\Wh'(\hat{\bZ}_2\tilde \pi))\cong H^1(\Wh'(\hat{\bZ}_2\pi)) \cong
0.$$
Therefore condition (iii) holds, and condition (iv) is immediate from the GAP computations in Appendix \ref{sec:apptwo},  which show that  $SK_1(\Zhat\pi) \cong \cy 2$. Condition (v) follows by applying Theorem \ref{thm:t1} and  Theorem \ref{thm:t2} (see List L2).

Now we get the corresponding Lyndon-Hochschild-Serre spectral sequences
with
\begin{equation*}
     \begin{aligned}
             d_2(\tilde X_5)  & = \tilde X_1 \tilde X_3 + \tilde
X_3\tilde X_4 + \tilde X_4^2 \\
             d_2(\tilde X_6)  & =  \tilde X_1^2 + \tilde X_1\tilde X_3
+ \tilde X_2^2  \\
             d_2(\tilde X_7)  & =  \tilde X_1^2 + \tilde X_2\tilde X_3
\\
             d_2(\tilde X_8)  & = \tilde X_1\tilde X_2  +   \tilde
X_1\tilde X_3.
     \end{aligned}
\end{equation*}
and
\begin{equation*}
     \begin{aligned}
             d_2(X_5)  & =  X_1X_2 + X_3X_4 + X_4^2\\
             d_2(X_6)  & =  X_1^2 + X_1X_3 + X_2^2\\
             d_2(X_7)  & =  X_1^2 + X_2X_3\\
     \end{aligned}
\end{equation*}
Assume that $I$ denotes the ideal generated by $d_2(X_5), d_2(X_6),
d_2(X_7)$ that is closed under Steenrod operations.  Notice that
$\alpha^*$ sends $d_2(X_5)$ to $d_2(\tilde X_5) + d_2(\tilde X_8)$.
Hence
$$\theta = [X_1X_2 + X_1X_3]\in H^2(\pi ;\cy 2)$$
is the cohomology class that corresponds to the central extension $$1
\to \sigma \to \tilde \pi \to \pi \to 1.$$ Now take
$$z=[X_3X_4]\in H^2(\pi ;\cy 2).$$
Then $\theta \cup z \neq 0$ in $H^4(\pi ;\cy 2)$, and pick  $\hat x\in H_4(\pi;\cy 2)$ such that $\hat x \, \cap \,  (\theta \,\cup\, z)\neq 0$. Then $(\hat x \, \cap \,  \theta)\, \cap \,  z=1$, and for  the class $\hat z = \hat x \, \cap \, \theta$  we have  
$ \hat z\, \cap \, z=\hat z \, \cap \,  [X_3X_4]=1 $.

\smallskip
    Now $H_2(\pi;  \bZ)$ has exponent two  (by a GAP computation) and $\mathrm{Sq}^1(z)\neq 0$ (by using the GAP package SINGULAR), so 
    $$0 \neq \beta^*(z) \in H^3(\pi; \bZ) \cong \Hom_\bZ(H_2(\pi, \bZ), \bbQ/\bZ).$$
 By duality, there exists a unique class $\hat z_1\in H_2(\pi; \bZ)$ such that $i_*(\hat z_1) = \hat z$.

. 
Next we claim that the class $\hat z_1$ is not in the subgroup $H_2^{\ab}(\pi) \subseteq H_2(\pi, \bZ)$. Let $q_*\colon H_2(\pi ;\bZ)\to H_2(\pi ^{\ab} ;\bZ) $ denote the natural map induced by the quotient $q\colon \pi \to \pi^{\ab}$. Then we have the formula
$$H_2(\pi^{\textup{ab}};\bZ)=\bigoplus_{1\leq i<j\leq 4}\cy 2\cdot e_{ij} \cong(\cy 2)^6.$$
The image of $q_*$  is
generated by $\{ e_{12}+e_{34}, e_{14}, e_{24}\}$, and only the basis elements  $e_{14}$ and $e_{24}$
are in the image under $q_*$ of $H_2^{\textup{ab}}(\pi)$,  since $\la x_1, x_4\ra \cong \cy 4 \times \cy 4$ and $\la x_2, x_4\ra \cong \cy 4 \times \cy 4$, but $[x_1, x_2] = [x_3,x_4] = x_5$.

Since $z= [X_3X_4] \in H^2(\pi;\cy 2)$, we have $z = q^*(X_3X_4)$ and hence
 $$ q_*(\hat z) \, \cap \, (X_3X_4) =  \hat z \, \cap \, q^*(X_3X_4)=  \hat z \, \cap \, [X_3X_4] = 1.$$
 We have $(e_{12}+e_{34})\cap (X_3X_4)=1$, $e_{14}\cap (X_3X_4)=0$, $e_{24}\cap (X_3X_4)=0$. Hence $q_*(\hat z_1)=(e_{12}+e_{34}) +c_1e_{14} + c_2e_{24}$ for some constants $c_1$, $c_2$. Therefore, $\hat z_1  \notin H_2^{\ab}(\pi)$ and we have $\omega ([\hat z_1])\neq 0$ as required for condition (vii).
We refer to the webpage:

\smallskip
\url{https://users.fmi.uni-jena.de/cohomology/128web/128gp1376.html}
 
\smallskip
\noindent
for information about the cohomology ring of $\pi=SG[128,1376]$. 
\end{proof}

\begin{remark}\label{rem:moreex} The list of groups in Example \ref{ex:thmEgroups} is a possible source for further examples. Conditions (iii), (iv) and (v) in Definition \ref{def:kappalist} were checked explicitly in Section \ref{subsec:k2nonzero} for the group $\pi = SG[128, 1377]$ and its 2-extension $\wpi = SG[256, 8177]$. Using the methods above, it is not hard to verify the remaining conditions to conclude that $\kappa^s_4(\pi) \neq 0$, but $\kappa'_4(\pi) = 0$.
\end{remark}

\appendix
\section{Some GAP computations for $H^1(\Wh'(\Zhat G))$}\label{sec:appone}
The following code takes a group $G$ as input and it outputs
$H^1(\Wh'(\Zhat G))$. It uses the fact that
$$H^1(\Wh'(\hat{\bZ}_2G))\cong \frac{\{ g\in G_{\ab}  \vv
      g^2=1 \}}{\langle g \vv g\sim g^{-1}\rangle }\cong \frac{S}{C}$$
where $S=\{ g\in G  \vv
g^2\in[G,G] \}$  and $C=\langle g\in S \vv g\sim g^{-1}\text{ or } g\in
[G,G]\rangle$
Here the variable \texttt{Sq\_in\_Comm} is loaded with a list of
elements that generate $S$ and the variable \texttt{Conj\_to\_inverse}
is loaded with elements that generate the group $C$.

\medskip
\begin{lstlisting}[language=GAP,
	basicstyle=\ttfamily\small,
	frame=single,               
	breaklines=true,            
	columns=fullflexible,      
	keepspaces=true,             
	caption={Function for computing
		$H^1(\Wh'(\Zhat G))$} ]
H1_Wh_prime:=function(G)
local Comm_G,g,Sq_in_Comm,Conj_to_inverse;
      Comm_G:=CommutatorSubgroup(G,G);
      Sq_in_Comm:=[];
      Conj_to_inverse:=List(Comm_G);
      for g in G do
          if not(g^2 in Comm_G) then continue; fi;
          Add(Sq_in_Comm,g);
          if (g^(-1) in g^G) then Add(Conj_to_inverse,g); fi;
      od;
      return Group(Sq_in_Comm)/Group(Conj_to_inverse);
end;;
      \end{lstlisting}
Here we can obtain a list of groups $G$ with nonzero  $H^1(\Wh'(\Zhat
G))$ as follows:
\begin{lstlisting}[language=GAP,
	basicstyle=\ttfamily\small,
	frame=single,               
	breaklines=true,            
	columns=fullflexible,      
	keepspaces=true,             
	caption={Code for listing groups with  nonzero
		$H^1(\Wh'(\Zhat G))$} ]
for n in [2,4,8,16,32,64,128] do
   for i in [1..NumberSmallGroups( n )] do
      G:=SmallGroup(n,i);
      H:=H1_Wh_prime(G);
      if Size(H)>1 then Print("[",n,",",i,",",Rank(H1_Wh_prime(G)),"], "); fi;
   od;
od;
      \end{lstlisting}
The result of the above listing is given as below. Note that
$\texttt{[s,i,r]}$ being in the list means $H^1(\Wh'(\Zhat
SG[s,i]))\cong (\cy 2)^{\oplus r} $.
\begin{lstlisting}[language=GAP,
	basicstyle=\ttfamily\small,
	frame=single,               
	breaklines=true,            
	columns=fullflexible,      
	keepspaces=true,             
	caption={The output of above listing} ]
[32,4,1], [32,13,1], [64,3,1], [64,15,1], [64,27,1], [64,28,1], [64,46,1], 
[64,48,1], [64,57,1], [64,62,1], [64,63,1], [64,64,2], [64,78,1], [64,81,1], 
[64,82,3], [64,84,1], [64,106,1], [64,113,1], [64,162,1], [64,164,1],
[128,27,1], [128,43,1], [128,44,1], [128,45,1], [128,49,1], [128,56,1], 
[128,76,1], [128,85,1], [128,99,1], [128,100,1], [128,101,1], [128,104,1], 
[128,129,1], [128,130,1], [128,156,1], [128,158,1], [128,166,1], [128,172,1],
[128,174,1], [128,175,2], [128,177,1], [128,180,1], [128,183,1], [128,184,1], 
[128,294,1], [128,298,1], [128,300,1], [128,302,1], [128,340,1], [128,348,1],
[128,457,1], [128,458,1], [128,476,1], [128,477,2], [128,478,2], [128,479,1], 
[128,481,1], [128,484,1], [128,499,1], [128,502,1], [128,506,1], [128,539,1],
[128,549,1], [128,550,1], [128,551,1], [128,552,1], [128,553,1], [128,554,1], 
[128,555,1], [128,556,1], [128,557,1], [128,558,1], [128,559,1], [128,560,1], 
[128,561,1], [128,562,1], [128,563,1], [128,564,1], [128,565,1], [128,566,1], 
[128,567,2], [128,568,1], [128,569,1], [128,570,1], [128,571,2], [128,572,3],
[128,573,2], [128,576,1], [128,578,1], [128,585,1], [128,649,1], [128,651,1], 
[128,652,1], [128,656,1], [128,658,1], [128,659,1], [128,680,1], [128,762,1],
[128,766,1], [128,767,1], [128,772,1], [128,803,1], [128,804,1], [128,805,2], 
[128,806,2], [128,807,1], [128,808,1], [128,809,2], [128,810,1], [128,811,2],
[128,812,1], [128,813,1], [128,814,1], [128,816,1], [128,817,1], [128,824,1], 
[128,826,1], [128,828,1], [128,829,1], [128,831,1], [128,832,1], [128,833,2],
[128,834,1], [128,835,1], [128,836,3], [128,889,1], [128,895,1], [128,897,1], 
[128,898,1], [128,965,1], [128,967,1], [128,970,1], [128,971,1], [128,999,1],
[128,1011,1], [128,1013,1], [128,1014,2], [128,1015,1], [128,1122,1], 
[128,1126,1], [128,1132,3], [128,1133,1], [128,1215,1], [128,1235,1], 
[128,1236,1], [128,1303,1], [128,1602,1], [128,1639,1], [128,1650,1], 
[128,1818,1], [128,1821,1]
\end{lstlisting}

\section{Some GAP computations for $H^1(SK_1(\Zhat\pi))$}\label{sec:apptwo}
Given a group $2$-group $G$, the following function find elements $g$ in the commutator group $[G,G]$ such that the order of $g$ is $2$ and $ \langle g \rangle { \rightarrowtail  }
G  \overset{ \alpha }{ \twoheadrightarrow }  G/\langle g \rangle  $ is a
central group extension and $g$ is not in  $\Omega (\alpha)  = \left\{
\,[\widetilde{g}_1,\widetilde{g}_2]\,\, \left| \,\,
\widetilde{g}_1,\widetilde{g}_2\in G \textup{ and }
\alpha([\widetilde{g}_1,\widetilde{g}_2])=1\, \right. \right\}$. Hence by  Theorem \ref{thm:t1}, if we set $\pi := G/\langle g \rangle$,  we obtain examples of   finite $2$-groups of order $\leq 128$ with
$SK_1(\Zhat \pi)\neq 0$. The function also determines if the extra condition in Theorem \ref{thm:t2} is
satisfied.

\medskip
\begin{lstlisting}[language=GAP,
	basicstyle=\ttfamily\small,
	frame=single,               
	breaklines=true,            
	columns=fullflexible,      
	keepspaces=true,             
	caption={Function listing quotients $\pi$ of a group $G$ such that $SK_1(\Zhat \pi)\neq 0$} ]
Some_quotients_with_nonzero_SK1:=function(G)
local List_of_commutators_in_G, ZG_cap_CommG, g, h, List_of_subgroups,
List_with_extra_cond, extra_condition_holds, id_gr, h_tilde, h_tilde_set,
hom_from_cover, G_quotient_g, there_exist_lift;
List_of_commutators_in_G:=[];
for g in G do
 for h in G do
List_of_commutators_in_G:=Union(List_of_commutators_in_G,[g^(-1)*h^(-1)*g*h]);
 od;
od;
List_of_subgroups:=[];
List_with_extra_cond:=[];
ZG_cap_CommG:=Intersection(Center(G),CommutatorSubgroup(G,G));
for g in ZG_cap_CommG do
 if not(Order(g)=2) then continue; fi;
 if g in List_of_commutators_in_G then continue; fi;
 id_gr:=IdGroup(G/Group(g));
 if id_gr in List_of_subgroups then continue; fi;
 Add(List_of_subgroups,IdGroup(G/Group(g)));
 hom_from_cover:=NaturalHomomorphismByNormalSubgroup(G,Group(g));
 G_quotient_g:=Image(hom_from_cover);
 extra_condition_holds:=true;
 for h in G_quotient_g do
  if h^(-1) in h^G_quotient_g then
   h_tilde_set:=PreImages(hom_from_cover,h);
       there_exist_lift:=false;
       for h_tilde in h_tilde_set do
        if h_tilde^(-1) in h_tilde^G then
         there_exist_lift:=true;
         break;
        fi;
       od;
       extra_condition_holds := extra_condition_holds and there_exist_lift;
      fi;
  od;
  if extra_condition_holds then
      List_with_extra_cond:=Union(List_with_extra_cond,[id_gr]);
  fi;
od;
return [List_of_subgroups,List_with_extra_cond];
end;;
\end{lstlisting}
Then one can find the list L1 of all $2$-groups $\pi$ of  order up to 128, which satisfy the criterion of Theorem \ref{thm:t1}, and hence have  $SK_1(\Zhat\pi)$
is nonzero,  The list L2 contains those groups in L1 which also satisfy the extra condition in Theorem \ref{thm:t2}.

\medskip
\begin{lstlisting}[language=GAP,
	basicstyle=\ttfamily\small,
	frame=single,               
	breaklines=true,            
	columns=fullflexible,      
	keepspaces=true,             
	caption={Code for generating the lists L1 and L2 as above} ]
L1:=[];
L2:=[];
List_of_2_covers:=[];
for size_G in [8,16,32,64,128,256] do
for index_G in [1..NumberSmallGroups( size_G )] do
 G:=SmallGroup(size_G,index_G);
 Result_of_func:=Some_quotients_with_nonzero_SK1(G);
 List_with_SK1_nonzero:=Result_of_func[1];
 List_which_satisfy_extra_condition:=Result_of_func[2];
 if Size(List_with_SK1_nonzero)=0 then continue; fi;
 Add(List_of_2_covers,[size_G,index_G,List_with_SK1_nonzero]);
 L1:=Union(L1,List_with_SK1_nonzero);
 L2:=Union(L2,List_which_satisfy_extra_condition);
od;
od;
\end{lstlisting}
Here we display the results of the above  code.
\begin{lstlisting}[language=GAP,
	basicstyle=\ttfamily\footnotesize,
	frame=single,               
	breaklines=true,            
	columns=fullflexible,      
	keepspaces=true,            
	caption={The output of above listing for L1 and L2} ]
gap> L1;
[ [ 64, 149 ], [ 64, 150 ], [ 64, 151 ], [ 64, 170 ], [ 64, 171 ], [ 64, 172 ], 
[ 64, 177 ], [ 64, 178 ], [ 64, 182 ], [ 128, 36 ], [ 128, 37 ], [ 128, 38 ], 
[ 128, 39 ], [ 128, 40 ], [ 128, 41 ], [ 128, 138 ], [ 128, 139 ], [ 128, 144 ], 
[128, 145 ], [ 128, 227 ], [ 128, 228 ], [ 128, 229 ], [ 128, 242 ], [128, 243 ], 
[ 128, 244 ], [ 128, 245 ], [ 128, 246 ],[ 128, 247 ], [ 128, 265 ], [ 128, 266 ], 
[ 128, 267 ], [ 128, 268 ], [128, 269 ], [ 128, 287 ], [ 128, 288 ], [ 128, 289 ], 
[ 128, 290 ], [ 128, 291 ], [ 128, 292 ], [ 128, 293 ], [ 128, 301 ], [ 128, 324 ], 
[ 128, 325 ], [ 128, 326 ], [ 128, 417 ], [ 128, 418 ], [ 128, 419 ], [ 128, 420 ], 
[ 128, 421 ], [ 128, 422 ], [ 128, 423 ], [ 128, 424 ], [ 128, 425 ], [ 128, 426 ],
[ 128, 427 ], [ 128, 428 ], [ 128, 429 ], [ 128, 430 ], [ 128, 431 ], [ 128, 432 ], 
[ 128, 433 ], [ 128, 434 ], [ 128, 435 ], [ 128, 436 ], [ 128, 446 ], [ 128, 447 ], 
[ 128, 448 ], [ 128, 449 ], [ 128, 450 ], [ 128, 451 ], [ 128, 452 ], [ 128, 453 ], 
[ 128, 454 ], [ 128, 455 ], [ 128, 541 ], [ 128, 543 ], [ 128, 568 ], [ 128, 570 ], 
[ 128, 579 ], [ 128, 581 ], [ 128, 626 ], [ 128, 627 ], [ 128, 629 ], [ 128, 667 ], 
[ 128, 668 ], [ 128, 670 ], [ 128, 675 ], [ 128, 676 ], [ 128, 678 ], [ 128, 691 ], 
[ 128, 692 ], [ 128, 693 ], [ 128, 695 ], [ 128, 703 ], [ 128, 704 ], [ 128, 705 ], 
[ 128, 707 ], [ 128, 724 ], [ 128, 725 ], [ 128, 727 ], [ 128, 950 ], [ 128, 951 ], 
[ 128, 952 ], [ 128, 975 ], [ 128, 976 ], [ 128, 977 ], [ 128, 982 ], [ 128, 983 ],
[ 128, 987 ], [ 128, 1345 ], [ 128, 1346 ], [ 128, 1347 ], [ 128, 1348 ], 
[ 128, 1349 ], [ 128, 1350 ], [ 128, 1351 ], [ 128, 1352 ], [ 128, 1353 ], 
[ 128, 1354 ], [ 128, 1355 ], [ 128, 1356 ], [ 128, 1357 ], [ 128, 1358 ], 
[ 128, 1359 ], [ 128, 1360 ], [ 128, 1361 ], [ 128, 1362 ], [ 128, 1363 ], 
[ 128, 1364 ], [ 128, 1365 ], [ 128, 1366 ], [ 128, 1367 ], [ 128, 1368 ], 
[ 128, 1369 ], [ 128, 1370 ], [ 128, 1371 ], [ 128, 1372 ], [ 128, 1373 ], 
[ 128, 1374 ], [ 128, 1375 ], [ 128, 1376 ], [ 128, 1377 ], [ 128, 1378 ], 
[ 128, 1379 ], [ 128, 1380 ], [ 128, 1381 ], [ 128, 1382 ], [ 128, 1383 ], 
[ 128, 1384 ], [ 128, 1385 ], [ 128, 1386 ], [ 128, 1387 ], [ 128, 1388 ], 
[ 128, 1389 ], [ 128, 1390 ], [ 128, 1391 ], [ 128, 1392 ], [ 128, 1393 ], 
[ 128, 1394 ], [ 128, 1395 ], [ 128, 1396 ], [ 128, 1397 ], [ 128, 1398 ], 
[ 128, 1399 ], [ 128, 1544 ], [ 128, 1545 ], [ 128, 1546 ], [ 128, 1547 ], 
[ 128, 1548 ], [ 128, 1549 ], [ 128, 1550 ], [ 128, 1551 ], [ 128, 1552 ],
[ 128, 1553 ], [ 128, 1554 ], [ 128, 1555 ], [ 128, 1556 ], [ 128, 1557 ], 
[ 128, 1558 ], [ 128, 1559 ], [ 128, 1560 ], [ 128, 1561 ], [ 128, 1562 ], 
[ 128, 1563 ], [ 128, 1564 ], [ 128, 1565 ], [ 128, 1566 ], [ 128, 1567 ], 
[ 128, 1568 ], [ 128, 1569 ], [ 128, 1570 ], [ 128, 1571 ], [ 128, 1572 ], 
[ 128, 1573 ], [ 128, 1574 ], [ 128, 1575 ], [ 128, 1576 ], [ 128, 1577 ], 
[ 128, 1783 ], [ 128, 1784 ], [ 128, 1785 ], [ 128, 1786 ], [ 128, 1864 ], 
[ 128, 1865 ], [ 128, 1866 ], [ 128, 1867 ], [ 128, 1880 ], [ 128, 1881 ], 
[ 128, 1882 ], [ 128, 1893 ], [ 128, 1894 ], [ 128, 1903 ], [ 128, 1904 ], 
[ 128, 1924 ], [ 128, 1925 ], [ 128, 1926 ], [ 128, 1927 ], [ 128, 1928 ], 
[ 128, 1929 ], [ 128, 1945 ], [ 128, 1946 ], [ 128, 1947 ], [ 128, 1948 ], 
[ 128, 1949 ], [ 128, 1950 ], [ 128, 1951 ], [ 128, 1966 ], [ 128, 1967 ], 
[ 128, 1968 ], [ 128, 1969 ], [ 128, 1970 ], [ 128, 1971 ], [ 128, 1972 ], 
[ 128, 1983 ], [ 128, 1984 ], [ 128, 1985 ], [ 128, 1986 ], [ 128, 1987 ],
[ 128, 1988 ] ]
gap> L2;
[ [ 128, 287 ], [ 128, 288 ], [ 128, 290 ], [ 128, 568 ], [ 128, 579 ], [ 128, 667 ], 
[ 128, 668 ], [ 128, 670 ], [ 128, 676 ], [ 128, 692 ], [128, 693 ], [ 128, 703 ], 
[ 128, 704 ], [ 128, 725 ], [ 128, 1375 ], [ 128, 1376 ], [ 128, 1377 ], 
[ 128, 1547 ], [ 128, 1549 ], [ 128, 1576 ] ]
\end{lstlisting}

Then one can compute the actual $SK_1(\Zhat \pi)$  by the following
function.  

\medskip
\begin{lstlisting}[language=GAP,
	basicstyle=\ttfamily\small,
	frame=single,               
	breaklines=true,            
	columns=fullflexible,      
	keepspaces=true,
	caption={Function that computes  $SK_1(\Zhat \pi)$ for a given group $\pi$}]
SK1:=function(G)
local hom_from_SC_to_G,SC,K,Comm_SC,List_of_comm,s1,s2,h,SC_b,K_b,iso;
hom_from_SC_to_G:=EpimorphismSchurCover(G);
SC_b:=Source(hom_from_SC_to_G);
K_b:=Kernel(hom_from_SC_to_G);
iso := IsomorphismPermGroup( SC_b );
SC := Image( iso );
K:=Image( iso , K_b);
Comm_SC:=CommutatorSubgroup(SC,SC);
List_of_comm:=[];
for s1 in SC do
for s2 in SC do
h:=s1*s2*(s1^(-1))*(s2^(-1));
if not(h in List_of_comm) and (h in K) then
Add(List_of_comm,h);
fi;
od;
od;
return IdGroup(Intersection(Comm_SC,K)/Group(List_of_comm));
end;;
\end{lstlisting}

\section{Some GAP code to test Conjecture \ref{conj:subquotients}}\label{app:appthree}
\begin{lstlisting}[language=GAP,
	basicstyle=\ttfamily\small,
	frame=single,               
	breaklines=true,            
	columns=fullflexible,      
	keepspaces=true,
	caption={Code for finding examples that satisfy the hypothesis of the conjecture.}]
LoadPackage("HAP");
for n in [2,4,8,16,32] do
 for i in [1..NumberSmallGroups( n )] do
  G:=SmallGroup(n,i);
  if IsAbelian(G) then continue; fi;
  Sb_G:=List(ConjugacyClassesSubgroups(G), Representative);
  Normal_Sb_G:=Filtered(Sb_G, x -> IsNormal(G,x));
  Cyclic_Quo:=Filtered(Normal_Sb_G, x -> IsCyclic(G/x));
  conj_G:=List(ConjugacyClasses(G),Representative);
  for N in Cyclic_Quo do
   if  Size(G)<4*Size(N) then continue; fi;
   W:=Filtered(Cyclic_Quo, x -> IsSubgroup(x,N) and Size(G)=2*Size(x))[1];
   T:=Filtered(Cyclic_Quo, x -> IsSubgroup(x,N) and Size(x)=2*Size(N))[1];
   p:=NaturalHomomorphismByNormalSubgroup(G,N);
   G_N:=Image(p);
   poi:=GroupHomomorphismByFunction(W,G_N,x->Image(p,x));;
   H_2_poi:=GroupHomology(poi,2,2);
   if Size(Image(H_2_poi))=1 then continue; fi;
   H_4_poi:=GroupHomology(poi,4,2);
   if Size(Image(H_4_poi))>1 then continue; fi;
   H_4_p:=GroupHomology(p,4,2);
   if Size(Image(H_4_p))=1 then continue; fi;
   conj_cl:=Filtered(conj_G,x -> (x in T) and not(x in N));
   Print("\n","A sequence N < T <= W < G = SG[",n,",",i,"]",
     " that satisfy hypothesis of conjecture is \n", [N,T,W,G],"\n",
     "For this sequence the number of conjugacy classes in T-N is ",
     Size(conj_cl),"\n", "They are represented by the elements: ",
     conj_cl,"\n"); 
  od;  
 od;
od;
\end{lstlisting}
The GAP code above shows that smallest groups that contains such a increasing sequence are $SG[32,8]$, $SG[64,13]$ and $SG[64,14]$. In these groups, each such increasing sequence of subgroups $T - N$ has an odd number of $\pi$-conjugacy classes. However, this is not new information about $\kappa_4'$ for these groups, since $H^1(\Wh'(\Zhat\pi)) = 0$.

%%%%%%%%%%%%%%%%%%%%%%%
%%%%%%%%%%%%%%%%%%%%%%%%%%
%\bibliographystyle{ih}
%\bibliography{kappa}
%\end{document}
%%%%%%%%%%%%%%%%%%%%%%%%%%%%%
\providecommand{\bysame}{\leavevmode\hbox to3em{\hrulefill}\thinspace}
\providecommand{\MR}{\relax\ifhmode\unskip\space\fi MR }
\providecommand{\MRhref}[2]{%
  \href{http://www.ams.org/mathscinet-getitem?mr=#1}{#2}
}
\providecommand{\href}[2]{#2}

\end{document}